\documentclass[12pt]{amsart}
\usepackage{amsmath,amsthm,amsfonts,amssymb,latexsym, mathrsfs}
\usepackage{graphicx, xcolor} 

\usepackage{hyperref}

\title[Characters and Sylow 3-subgroup Generation]{Characters and The Generation of Sylow 3-subgroups For Almost Simple Groups}

\author{Eden Ketchum}
\thanks{ This work was partially supported by the U.S. National Science Foundation, Award No. DMS-2439897}

\usepackage{ marvosym }
\usepackage{amssymb}
\usepackage{amsmath}
\usepackage{mathtools}    
\usepackage{xcolor}   \usepackage{ wasysym }
\usepackage{mathrsfs}
\usepackage{yfonts}
\usepackage{stmaryrd}

\usepackage{amsfonts}
\usepackage{fancyhdr}
\usepackage{comment}
\usepackage{array}
\usepackage[utf8]{inputenc}
\usepackage[a4paper, top=2.5cm, bottom=2.5cm, left=2.2cm, right=2.2cm]%
{geometry}

\usepackage{blindtext}
\usepackage{amssymb}
\usepackage{xcolor}

\usepackage{amsthm}

\usepackage{commath}
\usepackage{times}
\usepackage{amsmath}
\usepackage{changepage}
\usepackage{amssymb}
\usepackage{graphicx}
\usepackage{chngcntr}
\usepackage{stackengine}
\usepackage{yfonts}
\usepackage{ marvosym }

\usepackage{thmtools}
\usepackage{enumitem,kantlipsum}
\usepackage{thmtools}
\usepackage{tikz-cd}
\usepackage[normalem]{ulem}

\newcommand{\Z}{\mathbb{Z}}
\newcommand{\Q}{\mathbb{Q}}
\newcommand{\nor}{\trianglelefteq}
\newtheorem*{thmA}{Theorem A}
\newcommand{\Aut}{\operatorname{Aut}}

\usepackage{setspace} 

\newtheorem{theorem}{Theorem}[section]
\newtheorem{lemma}[theorem]{Lemma}
\newtheorem{corollary}[theorem]{Corollary}
\newtheorem{proposition}[theorem]{Proposition}

\newcommand{\syl}{\mathrm{Syl}}
\newcommand{\Sl}[2]{\operatorname{SL}}
\newcommand{\SL}{\operatorname{SL}}
\newcommand{\Irr}{\operatorname{Irr}}
\newcommand{\Syl}{\operatorname{Syl}}

\newcommand{\Gl}{\operatorname{GL}}
\newcommand{\GL}{\operatorname{GL}}

\newcommand{\PSl}[2]{\operatorname{PSL}}
\newcommand{\PSL}{\operatorname{PSL}}
\newcommand{\semi}{\rtimes}
\newcommand{\normal}{\trianglelefteq}
\newcommand{\PSp}{\operatorname{PSp}}

\begin{document}

\begin{abstract}
 Given an almost simple group  $A$, we algorithmically show that the character table of $A$ determines whether or not the  Sylow 3-subgroups of $A$ are 2-generated. We show this property is equivalent to a condition involving the Galois action on characters in the principal $3$-block. This result would be a consequence of the Alperin-McKay-Navarro conjecture.
\end{abstract}
 \maketitle
\section{introduction}

A popular topic in the representation theory of finite groups is so called local/global connections, which seek to describe when local properties, such as Sylow subgroup structure, predict global properties, such as aspects of the character table, and vice versa. An example of such a result is \cite[Theorem A]{RSV20} in which  Rizo, Schaeffer Fry and Vallejo show that finite groups having a cyclic Sylow $p$-subgroup is equivalent to the group's character table having a  property involving the action of a specific Galois automorphism on the principal $p$-block for $p=2,3$ (note that for $p> 3$ a weaker result is given in \cite{Val23} and a potential generalization to all primes is \cite[Question 1.5]{HMM22}). In \cite{NRSV21}  Navarro,  Rizo, Schaeffer Fry, and Vallejo characterized when a group has 2-generated Sylow 2-subgroups, again in terms  of the action of a specific Galois automorphism on the principal $p$-block. It was also shown by  Moret\'o and  Sambale in \cite{MS22} that, for certain classes of groups, the character table determines whether or not a group has $2$-generated Sylow $p$-subgroups; however, the general case remains open and, further, an algorithm to determine this property for $p \ge 3$ has yet to be determined. For any finite group $G$ we let $\Irr_0(B_0(G))$ denote the set of all height zero characters in the principal $3$-block of $G$,  and   we let $\xi$ be a primitive  $|G|$th root of unity.  Then we  define $\sigma \in \operatorname{Gal}(\mathbb{Q}(\xi)/\Q)$  to be the Galois automorphism which fixes $3'$ roots of unity and raises $3$-power roots to the fourth power. Then the main result of our article is the following:

 \begin{thmA}\label{a}
Let $S$ be a non-abelian finite simple group, and  $A$  an almost simple group with  $S\le A \le \Aut(S)$ such that $3$ divides the order of $A$. Then, for $P\in\Syl_3(A)$ we have $|P:\Phi(P)| = 9$ if and only if the number of fixed points in $\Irr_{0}(B_0(A))$ under $\sigma$ is in $\{6,9\}$.
\end{thmA}

 We recall Burnside's Basis Theorem, which states that,  given a $p$-group $P$, if $|P/\Phi(P)| = p^a $, where $\Phi(P)$ denotes the Frattini subgroup of $P$, then $a$ is the size of a minimal generating set for $P$.  Thus, Theorem A states that Sylow $3$-subgroups of an almost simple group $A$ have minimal generating set of size $2$ if and only if the number of fixed points in $\Irr_{0}(B_0(A))$ under $\sigma$ is in $\{6,9\}$.
It is believed that the conclusion of Theorem A will hold for all finite groups, since, as was noted in the introduction of   \cite{NRSV21}, it is a consequence of the Alperin-McKay-Navarro conjecture.
The reader should note that Theorem A could be thought of as a ``Galois version" of \cite[Theorem 3.3]{sf}. We build off of the arguments in  \cite[Section 5]{sf} throughout. Additionally  \cite{mmsv} studies the action of $\sigma$ and is used frequently.
 We restrict our attention to almost simple groups because in upcoming work it is foreseen that the statement for general groups  will be reduced to showing the conjecture holds for almost simple groups as well as some additional statements on simple groups, which will be similar to  "Galois versions" of the statements in \cite[Theorem 3.1]{sf}.

 The structure of this paper will be as follows. In Section 2 we collect preliminary results and prove Theorem A for almost simple groups of socle isomorphic to an alternating group or sporadic simple group. In Section 3 we classify all almost simple groups with socle isomorphic to a simple group of Lie type that have  Sylow $3$-subgroup with minimal generating set of size $2$. In Section 4 we conclude the proof of Theorem A.

\section{Preliminaries and initial results}

\noindent\textbf{Notes on notation:} Throughout, $\sigma$ will denote the Galois automorphism that fixes $3'$ roots of unity and maps $3$-power roots of unity to their fourth power. For a group $G$ we let $k_0(B_0(G))$ denote the number of characters of $3'$ degree lying  in the principal 3-block of $G$ and we let $k_{0,\sigma}(B_0(G))$ denote the number of $\sigma$-fixed characters of $3'$ degree lying  in the principal 3-block of $G$. To denote the sets of these characters rather than the quantities we use $\Irr_{0}(B_0(G))$ and $\Irr_{0,\sigma}(B_0(G))$ respectively. Given   $N \nor G$ and $\theta \in \Irr(N)$ we let $\Irr_{0,\sigma}(B_0(G)|\theta)$ denote the set of all $\chi \in \Irr_{0,\sigma}(B_0(G))$ such that  $\theta$ is a constituent of $\chi_N$.  Also note that we say ``$i$-generated" to mean ``has a minimal generating set of size $i$''.
\newline

We now begin by compiling some results which will be useful throughout.

 \begin{lemma}\label{3prime}
        Let $N\normal G $  such that $3 \nmid|G/N|$, and let $\theta \in \Irr_{0,\sigma}(B_0(N))$. Then $\Irr_{0,\sigma}(B_0(G)|\theta)$ is nonempty. 
        
    \end{lemma}

\begin{proof}
    From \cite[Corollary 3.6 (ii)]{mmsv}  we have that there exists $\chi\in \Irr(B_0(G))$ such that $[\chi,\theta^G]\neq 0$ and $\chi$ is $\sigma$-fixed. Using \cite[Theorem 5.12]{Nav18}, we have that $3\nmid \chi(1)$, thus $\chi \in \Irr_{0,\sigma}(B_0(G))$ as desired.
\end{proof}

\begin{theorem}\label{ald}(Alperin-Dade) Let $N \nor G$ such that $3 \nmid |G/N|$. Further assume $G = N\mathbf{C}_G(P)$ for $P \in \syl_3(G)$. Then restriction defined a bijection between $\Irr_{0,\sigma}(B_0(G))$ and $\Irr_{0,\sigma}(B_0(N))$.

\end{theorem}
\begin{proof}
The work of  \cite{Alperin76}  and \cite{Dad77} give that restriction defines a bijection between $\Irr(B_0(G))$ and $\Irr(B_0(N))$. This and Lemma \ref{3prime} give the result.
\end{proof}

\begin{lemma}\label{idunno}
    Let $p$ be a prime and let $P$ and $Q$ be nontrivial $p$-groups such that $Q$ acts on $P$ by automorphisms. Define $R:= P\rtimes Q$. Then $[R:\Phi(R)] > [Q:\Phi(Q)]$.
\end{lemma}
\begin{proof}
    Let $\{M_1,...,M_k\}$ be the set of maximal subgroups of $Q$. Then we have $P\rtimes M_i$ is  a maximal subgroup of $R$ for each $1\le i \le k$. Thus, $\Phi(R) \le \bigcap (P \rtimes M_i) = P \rtimes \Phi(Q)$. Since $Q$ is a subgroup of $R$, there must be some maximal subgroup $M < R$ with $Q \le M$. It is clear that $P \nleq M $; therefore,  since $P \le  P \rtimes \Phi(Q)$, this gives that $\Phi(R) \le M \cap  (P \rtimes \Phi(Q))<  P \rtimes \Phi(Q)$. Thus, we have 
    $$
    [R:\Phi(R)] = [R: P \rtimes \Phi(Q)][ P \rtimes \Phi(Q): \Phi(R)] =
    [Q:\Phi(Q)][ P \rtimes \Phi(Q): \Phi(R)]. 
    $$ 
    Therefore, $[R:\Phi(R)]>[Q:\Phi(Q)]$.
\end{proof}

\begin{lemma}\label{wreath}
    Let $p$ be a prime and let $P$ be a $p$-group  of the form $C_{p^{a_1}}\wr...\wr C_{p^{a_i}}$, where $a_k \in \Z^+$ for $1 \le k \le i$. Then $P$ is $i$-generated. 
\end{lemma}
\begin{proof}
It is a well known fact that given $G = H\wr K$ we have $G/G' \cong H/H'\times K/K'$ (see for example \cite[Prob. 1.6.20]{RobinsonGroupTheory}). Then, since $P' \le \Phi(P)$, we have that $[P:\Phi({P})] = [(P/P'):\Phi(P/P')]$ and we obtain $|P/\Phi(P)|= p^i $. Therefore,  $P$ is $i$-generated.
\end{proof}

We now show that Theorem A holds in a few particular cases. 

\begin{proposition}\label{gapstuff}
Let $S$ be a sporadic simple group, $ ^2F_4(2)\,'$, $G_2(3)$, $\PSL_3(3)$ or the alternating group on $n$ letters for $n \le 6$. 
If $S\le A\le \Aut(S)$,
then Theorem A holds for $A$.
\end{proposition}

\begin{proof}
    The result follows from computation in \cite{gap}.
\end{proof}

\begin{proposition}\label{An}
    Theorem $A$ holds for almost simple groups with socle isomorphic to the alternating group $A_n$
\end{proposition}

\begin{proof}
    The case of $n\le 6$ was considered in Proposition \ref{gapstuff}, so we need only consider cases where $n \ge 7$ and  $\Aut(A_n) = S_n$. Further, note that $A_n$ and $S_n $ have the same Sylow $3$-subgroups. If $n = 3^i$, then the Sylow $3$-subgroups of $S_n$ are of the form $P_i:=C_3\wr C_3\wr...\wr C_3$ with $i $ copies of $C_3$. From Lemma \ref{wreath} we then see that $P_i$ is $i$-generated. For arbitrary $n$ we first take the $3$-adic decomposition $n = a_0 + 3a_1+3^2a_2+...+3^ka_k$, where each $a_i \in \{0,1,2\}$. Then, if $P$ is a Sylow $3$-subgroup of $S_n$, we have  $P\cong\prod P_i^{a_i}$, where $P_i \in \Syl_3(S_{3^i})$. From this  we have $P/P' = \prod P_i^{a_i}/(P_i')^{a_i}\cong C_3^{a_1 + 2a_2+...+ka_k}$. Therefore, we have that $P' = \Phi(P)$ and $P/\Phi(P) = C_3^{a_1 + 2a_2+...+ka_k} $. It then follows that $P$ is 2-generated if and only if $n \in \{6,7,8,9,10,11\}$. By \cite[Lemma 4.1]{mmsv}
    we have that all characters of $A_n$ are $\sigma$-fixed. We also have that all characters  of $S_n$ are $\sigma$-fixed since they are rational. Thus, it suffices to find all $n$ such that $k_0(B_0(A_n)) \in \{6,9\}$ and $k_0(B_0(S_n)) \in \{6,9\}$. This was done in \cite[Section 4]{sf}, which gives the result.
\end{proof}

\section{Groups of Lie Type with $2$-generated Sylow $3$-subgroups}

The goal of this section is to classify all almost simple groups with socle a simple group of Lie type that have $2$-generated Sylow $3$-subgroups.

\textbf{Our Setting:} Let $\mathbf{G}$ be a connected reductive algebraic group of simply connected type defined over the field $\overline{\mathbb{F}}_p$ for some prime $p$.  We consider the finite groups of Lie type $G = \mathbf{G}^F$ for some Steinberg endomorphism $F$. We primarily consider simple groups of the form $S = G/\mathbf{Z}(G)$ (Recall that $ ^2F(2)\,'$ was considered in Proposition \ref{gapstuff} so we may omit it here). Given such a simple group, which we call simple groups of Lie type, we denote $(\mathbf{G}_{ad})^F$ by $\tilde{S}$ (cf. \cite[Definition 9.14]{MT11}).  As we are considering almost simple groups, we discuss their structure here. Let $S$ be a simple group of Lie type. We see that $S \le \tilde{S} \le \Aut(S)$ and $\tilde{S}$ is referred to as the group of inner-diagonal automorphisms and $\Aut(S)/\tilde{S}$ has representatives that are field and graph automorphisms (See \cite[Theorem 24.24]{MT11}).  Using \cite[Theorem. 4.10.2]{GLS}, we have that a Sylow $3$-subgroup of $\Aut(S)$ can be decomposed as $(P_T\rtimes P_W) \rtimes \langle \tau_1,\tau_2, F_{p} \rangle$, where $P_T$ is a Sylow $3$-subgroup of a maximal torus, $T \le\tilde{S}$; $P_W$ is isomorphic to a Sylow $3$-subgroup of $N_{\tilde{S}}(T)/T$; $\tau_1$ and $\tau_2$ are (possibly trivial) graph automorphisms; and $F_{p}$ is the standard Frobenius map. We see further that $P_T\rtimes P_W$ is isomorphic to a Sylow $3$-subgroup of $\tilde{S}$. We also obtain a similar decomposition any $Q \in \syl_3(S)$. We have $Q \cong Q_T \rtimes Q_W$, where $Q_T$ is a Sylow $3$-subgroup of a maximal torus $T \le S$, and $Q_W$ is isomorphic to a Sylow $3$-subgroup of $N_{S}(T)/T$.

We first consider the situation for groups of Lie type defined in characteristic $p = 3$.

\begin{proposition}\label{def}
 Let $S$ be a simple group of 
 Lie type defined in characteristic 3 which is not isomorphic to one of the groups discussed in Proposition \ref{gapstuff}. If $S \le A \le \Aut(S)$, then $A$ does not have 2-generated Sylow $3$-subgroups.

\end{proposition}

\begin{proof}
 Let $S$ be as above and let $P \in \syl_3(S)$. Let $r$ denote the rank of $\mathbf{G}$ and $q$ denote the size of the field over which $S$ is defined, unless $S$ is one of the Ree groups, in which case we take $q$ such that $S =\,^2G_2(q^2)$ and $q^2 = 3^{2a+1}$ for some $a$. First note that \cite[Thm.(2)]{War66} alongside the proof of \cite[Theorem 5.7]{sf} imply that in the case of $S =\,^2G_2(q^2)$ $9 > |P/P'| = |P/\Phi(P)|$, so the result holds in this case.  Next we claim that for all remaining cases $|P/\Phi(P)| = 9$ if and only if $(r,q) \in \{(2,3),(1,9)\}$. 
   We may assume $ S\neq G_2(3)$ by Proposition \ref{gapstuff}. Thus, arguing as in the proof of  \cite[proposition 5.7]{sf}, we have that $P/P'$ is elementary abelian. Therefore, $P/P'=P/\Phi(P)$ and the same proof in loc. cit. gives $|P/P'|= 9$ if and only if $(r,q) \in \{(2,3),(1,9)\}$. Note that all cases  with $(r,q) \in \{(2,3),(1,9)\}$ give groups isomorphic to groups discussed Proposition \ref{gapstuff}.

   Now assume $S \le A \le \Aut(S)$ with $3$ dividing $ |A/S|$. Then, arguing as in the fifth paragraph of the proof of \cite[Proposition 5.7]{sf}, we see that the Sylow $3$-subgroups of $A$ are not $2$-generated.
\end{proof}

We now consider groups of Lie type defined in characteristic $p\neq 3$. We begin by classifying when such groups have $2$-generated Sylow $3$-subgroups. 
For the remainder of the article, we will denote by $\GL_n^+(q)$ the general linear group over a field of size $q$ (similarly $\SL_n^+(q)$ and $\PSL_n^+(q)$ will denote the related special linear and projective special linear groups respectively) and we will denote by $\GL_n^-(q)$ the general unitary group over a field of size $q^2$ (similarly $\SL_n^-(q)$ and $\PSL_n^-(q)$ will denote the related special unitary and projective special unitary groups  respectively). 

The following results on these groups will be useful throughout.

\begin{proposition}
    \label{theo:gl} Let $q$ be a power of a prime and $\epsilon \in \{+,-\}$ with $3|(q-\epsilon)$. Let $\widetilde{P}$ be a Sylow 3-subgroup of $\operatorname{GL}_n^{\epsilon}(q)$.  Then $|\widetilde{P}/\Phi(\widetilde{P})|\le 81$ if and only  $n\in \{1,2,3,4,5,6,9,10,27\}$.  Further, we have that $|\widetilde{P}/\Phi(\widetilde{P})| = 9$ if and only if $n \in\{2,3\}$.
\end{proposition}

\begin{proof}
      Let $\tilde{P}_0 = C_{3^a}$, where $a = (q-\epsilon)_3$. Then recursively define $\tilde{P}_i = \tilde{P}_{i-1} \wr C_3$. It follows from Lemma \ref{wreath} that $\tilde{P}_i$ is $(i+1)$-generated.  Let $n = a_o + a_13+a_23^2+...+a_t3^t$ be the 3-adic decomposition of $n$. Then, from the work of  \cite{weir}, if $\tilde{P}\in \Syl_3(\GL^{\epsilon}_n(q))$, we have   $ \tilde{P} \cong \displaystyle{\Pi_{i=0}^t\tilde{P}_i^{a_i}}$, where $\tilde{P}_i$ is as described above. It follows that $\tilde{P}'  = \displaystyle{\Pi_{i=0}^t\tilde{P}_i'^{a_i}}$, and, arguing as in the above discussion, we obtain $|\tilde{P}/\Phi{(\tilde{P})}| = 3^{a_o + 2a_1 + 3a_2...(t+1)a_{t}}$. Therefore, the stated result holds. 
\end{proof}

\begin{lemma}\label{lem:sl}
Let $\widetilde{P}$ be a Sylow $3$-subgroup of $\Gl_n^{\epsilon}(q)$. Further let $P$  be a Sylow $3$-subgroup of $\SL_n^{\epsilon}(q)$ and $Q$ a Sylow $3$-subgroup of $\PSL^{\epsilon}_n(q)$. Then the following hold. 
\begin{enumerate}
    \item $|\widetilde{P}:\Phi(\widetilde{P})|/  |P:\Phi(P)|\leq 3$.
    \item $|P:\Phi(P)|/ |Q:\Phi(Q)|\leq 3$
\end{enumerate}

\end{lemma}
\begin{proof}
     If $3\nmid(q-\epsilon)$, then $\tilde{P} = P \cong Q$ and the statement is trivial, so we may assume $3|(q-\epsilon)$.
     To show (1) we let $x$ be a generator of a Sylow $3$-subgroup of $C_{q-\epsilon}\le \mathbb{F}_{q^2}^\times$, and choose $\widetilde{P} \in \syl_3(\GL_n^{\epsilon}(q))$ such that $X:= \mathrm{diag}(x, I_{n-1})\in \widetilde P$.
    First, we note that $P= \widetilde{P}\cap \SL_n^{\epsilon}(q)$ is a Sylow 3-subgroup of $\SL_n^{\epsilon}(q)$. 
   Next, we note that, since $\langle X\rangle$ forms a complete set of coset representative for $\widetilde{P}/P$, every element of $\widetilde{P}$ can be written as an element of $P$ multiplied by an element of $\langle X\rangle$. Therefore, if $Y$ is a minimal generating set of $P$, then $Y\cup\{X\}$ is a generating set of $\widetilde{P}$. This is sufficient to show (1).
   
   For (2)  if $\pi:\SL_n^{\epsilon}(q) \rightarrow \PSL_n^{\epsilon}(q)$ is the usual projection map, we may choose $Q$ such that $Q  = \pi(P) \cong  P /(P \cap \mathbf{Z}(\SL_n(q))$. We also note that, since
    $\mathbf{Z}(\SL^{\epsilon}_n(q))$ is cyclic, $P \cap \mathbf{Z}(\SL^{\epsilon}_n(q)) =\langle A\rangle$ for some $A$. Let $\{A_1,...,A_n\}$ be a complete set of coset representatives for a minimal generating set of $Q$. Then $\{A_1,...,A_n,A\}$ must generate $P$. This is sufficient for the desired result.
   \end{proof}

\begin{proposition}\label{theo:sl}  Let $P$ be a Sylow 3-subgroup of $\SL_n^{\epsilon}(q)$, where  $q$ is a power of a prime such that $3|(q-\epsilon)$.  If $ 9 \le|P/\Phi(P)|\le 27$,  then $n\in \{3,4,5,6,9,10\}$. Furthermore, $|P/\Phi(P)| = 9$ if and only if $n \in \{3,4\}$. 
\end{proposition}

\begin{proof}
 First, we consider the case in which $3\nmid n$. By 
  order considerations we see that we can construct $P \in \Syl_3(\operatorname{SL}^{\epsilon}_n(q))$ to be the set
$$P := \left\{
\begin{pmatrix}
A&0\\
0&\det(A)^{-1}
\end{pmatrix}
|A \in \tilde{P}\right\}$$
for a fixed $\tilde P \in \Syl_3(\GL_{n-1}^{\epsilon}(q))$. Thus, a Sylow $3$-subgroup of $\SL_n^{\epsilon}(q)$ is isomorphic to a Sylow $3$-subgroup of $\GL^{\epsilon}_{n-1}(q)$ and the result follows from  the proof of Proposition \ref{theo:gl}.

Now assume $3|n$. Proposition \ref{theo:gl} and Lemma \ref{lem:sl} give that the only cases we need to consider are when $n\in \{3,6,9,27\}$. If $n = 3^i$, let $P_W = C_3\wr..\wr C_3$ with $i$ copies of $C_3$. In this case we see explicitly that $P = C_{3^a}^{3^{i}-1}\semi P_W $, where $3^a = (q-\epsilon)_3$. Applying Lemma \ref{idunno} and Lemma \ref{wreath} gives $|P/\Phi(P)| > |P_W/\Phi(P_W)| = 3^i $. Thus, if $n =\{9,27\}$ we have that $P$ is not $2$-generated and, if $n = 27$, $P$ is not $3$-generated.  If $n=3$, we have that $|P/\Phi(P)|\ge 9$ and we can explicitly find an order $2$ generating set, which gives $|P/\Phi(P)|= 9$. For $n= 6$  we instead take $P_W\cong C_3 \times C_3$ and the argument follows similarly to the case of $n = 9$ above.
\end{proof}

\begin{proposition}\label{theo:psl}
Let $q$ be a power of a prime that is distinct from 3 and let $Q$ be a Sylow $3$-subgroup of $\operatorname{\PSL}_n^{\epsilon}(q)$. If $3|(q-\epsilon)$, then $|Q/\Phi(Q)|= 9$ if and only if $n=3$ or $n= 4$. If instead $3|(q +\epsilon)$, then $|Q/\Phi(Q)|= 9$ if and only if $n \in \{4,5,6,7\}$.
\end{proposition}

\begin{proof}
     Assume $3|(q-\epsilon)$. Let $\pi:\SL_n^{\epsilon}(q) \rightarrow \PSL_n^{\epsilon}(q)$ be the usual projection map, and let $P \in \Syl_3(\SL^{\epsilon}_n(q))$. Define $Q:= P/(P\cap \mathbf{Z}(\SL^{\epsilon}_n(q))\ \cong \pi(P)\in \Syl_3(\PSL^{\epsilon}_n(q))$.  In cases where $3\nmid n$ we have that $ Q\cong P$, thus, from  Proposition \ref{theo:sl}, we have that  $Q$ is $2$-generated and if and only if $n = 4 $.

     Now assume $3|n$. Using Lemma \ref{lem:sl}, Proposition \ref{theo:sl}, and the fact that $|Q/\Phi(Q)|$ must be less than  $|P/\Phi(P)| $ since $Q$ is a homomorphic image of $P$, we obtain that the only possible choices of $n$ which could give 2-generated Sylow 3-subgroups are $n \in \{3,6,9\}$.

  Assume  $n=3$. By Proposition \ref{theo:sl} we see that $P$ is $2$-generated, so, since $Q$ is the homomorphic image of $P$, $Q$ must either be $2$-generated or cyclic. Assume by way of contradiction that $Q$ is cyclic. 
    We have $Q\cong P/Z_3$, where $Z_3 = P\cap \mathbf{Z}(\SL^{\epsilon}_n(q))$. Since $Z_3\subseteq \mathbf{Z}(P)$, we have that $P/\mathbf{Z}(P)$ is a quotient group of $Q$. Thus, if $Q$ is cyclic, then $P/\mathbf{Z}(P)$ is cyclic. It is a well known fact that $P/\mathbf{Z}(P)$ being cyclic implies that $P$ is abelian. This contradicts Proposition \ref{theo:sl}.

Now assume $n\in \{6,9\}$. Then, arguing as in the proof of Proposition \ref{theo:sl}, we have that $Q = (C_{3^a}^{8}\rtimes P_W)/Z_3$, where $Z_3$ is as above and $P_W = C_3\times C_3$ or $P_W= C_3\wr C_3$ for $n=6$ and $n = 9$ respectively. We see further that $P\cap \mathbf{Z}(\SL_n^{\epsilon}(q))\le C_{3^a}^{8}$ and $Q =  C_{3^a}^{8}/Z_3\rtimes P_W $. Applying Lemma \ref{idunno}  gives $|Q/\Phi(Q)| > 9$.

If we instead assume $3|(q+\epsilon)$, then we have $3\nmid [\Gl_n^\epsilon(q):\SL_n^\epsilon(q)]$ and $3\nmid |\mathbf{Z}(\SL_n^\epsilon(q))|$, so it is sufficient to consider Sylow $3$-subgroups of $\Gl_n^\epsilon(q)$. The work of \cite{weir} gives that $Q$ is isomorphic to a Sylow subgroup of $\Gl_{\lfloor n/2\rfloor}^\epsilon(q^2)$ and the proof of Proposition \ref{theo:gl} gives the desired result.
\end{proof}

\begin{proposition}\label{theo:classical}
Let $S$ be a simple group  of Lie type defined in  characteristic $p\neq 3$. Then $P\in \syl_3(S)$ is 2-generated if and only if $S$ is one of the following:
\begin{enumerate}
    \item[(i)] $\PSL^{\epsilon}_n(q)$ with   $n\in \{3,4\}$ and $3|(q-\epsilon)$
    \item[(ii)] $\PSL^{\epsilon}_n(q)$ with $n\in \{4,5,6,7\}$ and $3|(q+\epsilon)$
    \item[(iii)] $\PSp_{2n}(q)$  with $n=2$ or $n=3$
    \item[(iv)] $\operatorname{P}\Omega_{2n+1}(q)$  with $n=2$ or $n=3$
     \item[(v)] $\operatorname{P\Omega}_8^-(q)$
    \item[(vi)] $\operatorname{G}_2(q),\,^3\operatorname{D}_4(q)$ or $\,^2\operatorname{F}_4(q^2)'$
\end{enumerate}

\end{proposition}

\begin{proof}
If $S = \PSL^{\epsilon}_n(q)$, then the result is given by proposition \ref{theo:psl}, so we may assume $S\neq \PSL^{\epsilon}_n(q)$.
If $S = \PSp_{2n}(q)$ or $\operatorname{P\Omega}_{2n+1}(q)$ and   $3|(q-1)$, then \cite{weir} gives  that $P$ can be identified with a Sylow $3$-subgroup of $\GL_n(q)$, so, by Proposition \ref{theo:gl}, we  get that 
$P$ is $2$-generated if and only if $n = 2$ or $n=3$. If instead $3|(q+1)$, then loc. cit. gives us that $P$ is isomorphic to a Sylow $3$-subgroup of $\GL_{n}(q^2)$. Thus, proposition \ref{theo:gl} again gives us that $P$ is $2$-generated if and only if $n = 2$ or $n=3$.

Now let $S = \operatorname{P\Omega^{\epsilon}_{2n}(q)}$ with $\epsilon\in\{\pm1\}$. Again using \cite{weir} we have that Sylow 3-subgroups can be identified with those of either $\operatorname{P\Omega_{2n-1}(q)}$ or $\operatorname{P\Omega_{2n+1}(q)}$. Thus, the above paragraph gives that, if $n > 4$, we have that $P$ is not 2-generated. In the case of $n = 4$ we see from the order polynomials that, if $\epsilon =-1$, we have Sylow $3$-subgroups correspond to those of $\operatorname{P\Omega_{7}(q)}$, and, if $\epsilon = 1$,   Sylow $3$-subgroups correspond to those of $\operatorname{P\Omega_{9}(q)}$. Therefore, only in the case of $S= \operatorname{P\Omega}^-_8(q)$ do we have $2$-generated Sylow $3$-subgroups.

Now we consider exceptional groups of Lie type. Note that $3$ does not divide the order of the Suzuki groups $^2\operatorname{B}_2(q^2)$. Also note that the Ree Groups $^2\operatorname{G}_2(q^2)$ are only defined in characteristic $3$, so we need not consider either class of groups here. If $S = \operatorname{G}_2(q)$, by considering maximal subgroups discussed in \cite{Coo81} and \cite{Kle88} we obtain that Sylow $3$-subgroups are isomorphic to those of $\SL_3^{\epsilon}(q)$, where $\epsilon \in \{\pm 1\} $ such that $3|(q-\epsilon)$. In either case we have that any $P \in \Syl_3(\SL_3^{\epsilon}(q))$ is $2$-generated by Proposition \ref{theo:sl}.  If $S = \,^3\operatorname{D}_4(q)$, we have by \cite[Proposition 2.2]{DM87} that $P$ is an extension of a Sylow $3$-subgroup of $\SL_3^\epsilon(q)$ by an outer diagonal automorphism of order $3$, where $\epsilon \in \{\pm1\}$ such that  $3|(q-\epsilon)$. It is immediate that such a groups is not cyclic and we can explicitly find a generating set of order $2$; therefore, they must be $2$-generated. Next let $S = \,^2\operatorname{F}_4(q^2)'$ with $q^2 = 2^{2a+1}$. Then \cite[Main Theorem]{Mal90} 
implies that $S$ has Sylow $3$-subgroups isomorphic to those of $\operatorname{SU}_3(q^2)$, which are again $2$-generated by Proposition \ref{lem:sl}.

Lastly we consider the groups $\operatorname{F}_4(q),\operatorname{E}_6^{\epsilon}(q),\operatorname{E}_7(q),\operatorname{E}_8(q)$. The final paragraph of the proof of \cite[Proposition 5.17]{sf} gives the desired result.
\end{proof}

This concludes the discussion for simple groups, and we now shift our attention to almost simple groups.

\begin{proposition}\label{noncyclic}
     Let $S$ be a simple group of Lie type defined over a field of order $q$ with characteristic apart from $3$,
 and let $S\le A \le \Aut(S)$ such that $A/S$ has non-cyclic Sylow $3$-subgroups. Then $|P/\Phi(P)| >9 $.
 \end{proposition}

 \begin{proof}
      Let $A$ be such a group, let $Q\in \Syl_3(S)$ and let $P\in \Syl_3(A)$ such that $Q\le P$.  Let $\{M_1,..,M_k\}$ be the set of maximal subgroups of $P/Q$.  For any $M_i$  we have that the preimage of $M_i$ under the canonical projection map $\pi$ will be a maximal subgroup of $P$. By assumption we have $[P/Q:\Phi(P/Q)]\ge 9$, so $\bigcap\pi^{-1}(M_i)$ has index at least $9$ in $P$ and  $[P:\Phi(P)]\ge 9$. It is therefore sufficient to find a maximal subgroup of $P$ which does not contain this intersection. 
      
      Using  \cite[Theorem. 4.10.2]{GLS}  we see that there is a Sylow $3$-subgroup of $A$  contained in a group of  the form $(P_T\rtimes P_W)\rtimes \hat{A}$, where $P_T$ is a Sylow $3$-subgroup of a maximal torus in $T \le\tilde{S}$, $P_W$ is isomorphic to an element of $ \Syl_3(N_{\tilde S}(T)/T)$ and $\hat{A}$ is generated by field and graph automorphisms. We see further that any outer diagonal automorphisms in $P$ are contained in $P_T$.  This decomposition can be chosen such that the elements of $\hat A$ normalize $P_T$. Note that, from the structure description given in the preceding propositions, any $S$ with $P_W$ being trivial is one of $\PSL_2(q),\PSL^{\epsilon}_3(q)$ with $3|(q+\epsilon)$, $\PSL^{\epsilon}_4(q)$ with $3|(q+\epsilon)$, $\PSL^{\epsilon}_5(q)$ with $3|(q+\epsilon)$ or  $ \operatorname{PSp}_4(q)$. In all of these cases $\Aut(S)/S$ has cyclic Sylow $3$-subgroups, so we need not consider them here.  This gives that $(P_T\rtimes \hat A)\cap A$ is a proper subgroup of $A$. Let $M$ be a maximal subgroup of $P$ containing  $(P_T\rtimes \hat A)\cap A$. Then, $M$ does not contain $P_W$, since $\langle (P_T\rtimes \hat A)\cap A,P_W\rangle  = P$. Since $P_W\le Q$ we have  $P_W \le \bigcap\pi^{-1}(M_i)$.  This implies that $\bigcap\pi^{-1}(M_i)$ is not contained in $M$ as desired.
 \end{proof}

\begin{proposition}\label{2or3}
    Let $S = \PSL_2^{\epsilon}(q)$ or $S=\PSL^{\epsilon}_3(q)$, and further let $S \le A \le \Aut(S)$ such that $A/S$ has non-trivial, cyclic Sylow $3$-subgroups. Then $A$ has $2$-generated Sylow $3$-subgroups unless $S=\PSL^{\epsilon}_3(q)$, $3|(q-\epsilon)$, and $\langle F_0\rangle Q$ is a Sylow $3$-subgroup of $A$ for some $Q\in \Syl_3(S)$ and $F_0$ a field automorphism.
\end{proposition}

\begin{proof}
    Let $Q \in \Syl_3(S)$ and $ P \in \Syl_3(A)$ with $Q \le  P$. First, consider the case where $S=\PSL^{\epsilon}_2(q)$ or $3|(q+\epsilon)$ and $S= \PSL^{\epsilon}_3(q)$. In this case we can see explicitly that $Q$ is cyclic. Since $3\nmid |\tilde S/S|$, we can choose $P$ and $Q$ such that  $ P = Q \rtimes \langle F_0\rangle$, where $F_0$ is a field automorphism. From this it is clear that $ P$ is $2$-generated. 

     Now let $S=\PSL^{\epsilon}_3(q)$ with $3|(q-\epsilon)$. First assume a Sylow $3$-subgroup of $A/S$ is generated by an outer diagonal automorphism or the product of a field automorphism with an outer diagonal automorphism. Since $Q \le  P$ is not cyclic, $ P$ cannot be cyclic. In the case that $A/S$ has a Sylow $3$-subgroup generated by an outer diagonal automorphism we see that $P\in \syl_3(\operatorname{PGL}_3^\epsilon(q))$ and is a homomorphic image a Sylow $3$ subgroup of $\GL_3^\epsilon(q)$, which is $2$-generated by \ref{theo:gl}, and must therefore be $2$-generated. If instead $A/S$ is generated by the product of a field automorphism $F_0$ with an outer diagonal automorphism $\delta$ we let  $a$ be a generator of a Sylow $3$-subgroup of $C_{q-\epsilon}\le \mathbb{F}_{q^2}^\times$ and note that $\{\mathrm{diag}(a,a^{-1},1),\delta F_0\}$ generates $P$. Thus, $P$ must be $2$-generated.

    Lastly, we consider the case where  a  Sylow $3$-subgroup of $A/S$ is generated by  a field automorphism. As in the preceding proposition, we have that  $P$ can be decomposed as $ (P_T\rtimes P_W)\rtimes \langle F_0\rangle \cong P_T\rtimes (P_W\rtimes \langle F_0\rangle)  $, where $F_0$ is a field automorphism. Applying Lemma \ref{idunno} to the subgroups $P_T$  and $P_W\rtimes \langle F_0\rangle$ gives the desired result. 
 \end{proof}

\begin{proposition}\label{alm} Let $S$ be a simple group of Lie type defined over a field of order $q$ with characteristic apart from $3$.
Assume $S\neq \PSL_2(q)$ and $S \neq \PSL^{\epsilon}_3(q) $. Let $S\le A \le \Aut(S)$ such that $A/S$ has a non-trivial Sylow 3-subgroup. Then $|P/\Phi(P)| >9 $ for all $P \in \Syl_3(A)$.

\end{proposition}

\begin{proof}
 Let $Q \in \Syl_3(S)$ and $P \in \Syl_3(A)$. Then, $ Q$ is of the form $Q_T\rtimes Q_W$, where $Q_T$ and $Q_W $ are as discussed in the setting. Using Proposition \ref{noncyclic} we need only consider cases where $A/S$ has cyclic Sylow $3$-subgroups.

First, consider the case in which $P = Q \rtimes\langle\varphi\rangle$, where $\varphi = F_{q_0}$ is a field automorphism of $3$-power order, or $S = \operatorname{D}_4(q)$ and $\varphi \in \{\tau F_{q_0},\tau\}$ and  $\tau$ is  a triality graph automorphism (note that the argument in \cite[Proposition 5.18]{sf} applies to this and the following case, but we include a proof for completeness). We can choose the decomposition $Q = Q_T\rtimes Q_W$ such that   $\varphi$ centralizes $Q_W$. In all cases aside from $S =\operatorname{PSp}_4(q)$, $S=\PSL^{\epsilon}_4(q)$ with $3|(q+\epsilon)$, and $S=\PSL^{\epsilon}_5(q)$ with $3|(q+\epsilon)$  we have that $Q_W$ is non-trivial and we have that $Q_W\times \langle\varphi\rangle$ is a non-cyclic subgroup of $P$. Thus, if we take $S$ to be not one of these exceptions, applying Lemma \ref{idunno} gives the desired result. Further note that, in the case of $S= \operatorname{E}_6^{\epsilon}(q)$,  from \cite[Table 3.2]{GM20} we see that the possibilities for relative Weyl groups of minimal  $1$-split and minimal  $2$-split Levi subgroups are the Shephard--Todd Groups $\operatorname{G}_{28}$ and $\operatorname{G}_{35}$.  We verify that these groups have non-cyclic Sylow $3$-subgroups using their implementation in  \cite{BCP97}. Thus, since $3 |(q\pm1)$, we see that \cite[Theorem 4.10.2]{GLS} gives that $Q_W$ is isomorphic to to Sylow $3$-subgroup of the relative Weyl group of a minimal  $1$-split or minimal  $2$-split Levi subgroup. This gives that $Q_W$ is non-cyclic. Thus, we have $Q_W\times \langle\varphi\rangle $ being at least $3$-generated and $P$ being at least $4$-generated by Lemma \ref{idunno}.

We now consider the three exceptions listed in the previous paragraph. From the proofs of Proposition \ref{theo:classical}, Proposition \ref{theo:gl} and Proposition \ref{theo:psl} we see that $Q \cong C_{3^a}\times C_{3^a} $, where $3^a = (q^2-1)_3$ and  $P \cong Q \rtimes \langle F_{q_0}\rangle$. Let $M$ be a maximal subgroup of $P$. Then, if $M\cap Q$ is proper, it is maximal in $Q$, since the diamond isomorphism theorem gives $[Q:M\cap Q] = [P:M] = 3$. Similarly we have $M\cap \langle F_{q_0}\rangle$ is either non-proper or maximal in $\langle F_{q_0}\rangle$. Thus, it is sufficient to consider the size of a minimal generating set  of $Q/\Phi(Q) \rtimes (\langle F_{q_0}\rangle/\Phi(\langle F_{q_0}\rangle)) \cong (C_3\times C_3)\rtimes C_3$. Since the field automorphism acts coordinate by coordinate on the $Q_T\cong C_{3^a}\times C_{3^a}$, it follows that  $P/\Phi(P) \cong Q/\Phi(Q) \rtimes (\langle F_{q_0}\rangle/\Phi(\langle F_{q_0}\rangle))$ is isomorphic to $C_3\times C_3\times C_3$ and $P$ is  $3$-generated.

 Now let $S = \PSL^{\epsilon}_n(q)$ and $P = Q \langle \delta\rangle$, where $\delta$ is an outer diagonal automorphism of $3$-power order. We then have $3|\mathrm{gcd}(n,q-\epsilon)$. 
 If $\pi:\GL^{\epsilon}_n(q)\rightarrow \operatorname{PGL}^{\epsilon}_n(q)$ is the usual projection map, we then have that $P$ is the image under $\pi$ of a Sylow $3$-subgroup of some group $M$ with $\SL^{\epsilon}_n(q)\le M\le \GL^{\epsilon}_n(q)$. If $\hat P \in \Syl_3(\GL^{\epsilon}_n(q))$ and $\hat{Q} \in \Syl_3(M)$,  then, arguing as in Lemma \ref{lem:sl}, we see that $[\hat P : \Phi(\hat{P})]/[\hat{Q}:\Phi(\hat{Q})] \le 3$ and $[\hat Q : \Phi(\hat{Q})]/[Q:\Phi(Q)] \le 3$.
     By Proposition \ref{theo:gl} we then ascertain that the only choices of $n$ that could allow for $2$-generated Sylow subgroups are $n \in \{3,6,9,27\}$. The case of $n = 3 $ was discussed in Proposition \ref{2or3}, so we need only consider $n \in \{ 6,9,27\}$. As in the above cases we have $P$ can be decomposed as $P_T\rtimes P_W$ and further we have that $P_W$ is isomorphic to  $ C_3\times C_3,  C_3\wr C_3$ and $ C_3\wr C_3\wr C_3$ for $n=6,n=9$ and $n=27$ respectively. Since $|P_W/\Phi(P_W)|\ge 9$ in all  cases, applying Lemma \ref{idunno} gives the desired result.

 Now we consider the case where $S = \PSL^{\epsilon}_n(q)$ and $P = Q  \langle \delta F_{q_0}\rangle$, where $\delta$ is a non-trivial outer diagonal automorphism of $3$-power order and $F_{q_0}$ is a field automorphism.   Then $P \le \hat P := \tilde Q_T\rtimes Q_W \rtimes \langle F_{q_0} \rangle $, where $\tilde Q_T\rtimes Q_W$ is a Sylow $3$-subgroup of $\tilde{S}$. We can choose this decomposition such that $F_{q_0}$ commutes with $Q_W$, so we have that $\hat{P} \cong \tilde Q_T\rtimes \langle F_{q_0} \rangle \rtimes Q_W $, and it follows that $P = ((\tilde Q_T\rtimes \langle F_{q_0} \rangle)\cap P)\rtimes Q_W $, since $P_W\le P$. Since $n > 3$  and $3|n$, we have $|Q_W/\Phi(Q_W)|\ge 9$ and Lemma \ref{idunno} gives the result.

 The last case to consider is where $S = \operatorname{E}^{\epsilon}_6(q)$ and $P$ is not of the form $Q \rtimes \langle F_{q_0}\rangle$ with $F_{q_0} $ a field automorphism. We have that either $P = Q  \langle \delta F_{q_0}\rangle$ for some field automorphism $F_{q_0}$ or $P= Q  \langle \delta\rangle$ with $\delta$  an outer diagonal automorphism. In the first case we have  $P = Q \rtimes \langle \delta F_{q_0}\rangle$ is a subgroup of the group  $\hat P :=\tilde Q\rtimes \langle F_{q_0}\rangle$, where $\tilde{Q}$ is a Sylow $3$-subgroup of $\tilde{S}$. Note that $[\hat P: P] \le 3$ and $\Phi(\hat P) \subseteq P $. Further, note that, for any $ M \le \hat P$ maximal with $M \neq P$, we have $M \cap P $ is a maximal subgroup of $P$. Thus, $\Phi(P) \le \Phi(\hat P)$. Therefore, we have

 $$
 [\hat P:\Phi(\hat P)] \le 3[P: \Phi(\hat P)] \le 3[P:\Phi(P)]
 $$

From the second paragraph of this proof we have $|\hat P:\Phi(\hat P)| \ge 81$, so this gives $|P: \Phi( P)|\ge 27$ as desired.  In the latter case $P$ is a Sylow $3$-subgroups of $\tilde{S}$ and  we can again obtain a decomposition of the form $ P =P_T \rtimes P_W $. We see that it is sufficient to show that  $P_W$ has a minimal generating set of size at least $2$, since this gives that $P$ is not $2$-generated by Lemma \ref{idunno}. Arguing as in the second paragraph of this proof gives $P_W$ having minimal generating set of size at least 2.
 \end{proof}

 We conclude this section by compiling a list of all almost simple groups with $2$-generated Sylow $3$-subgroups. 
 \begin{corollary}
     Let $S$ be a simple group and  $S\le A \le \Aut(S)$.

 \begin{enumerate}
\item[(I)] If $3 \nmid |A/S|$ then $A$ has $2$-generated Sylow $3$-subgroups if and only of $S$ is isomorphic to  one of the following:

 \begin{itemize}
 \item $\PSL_3(3)$
    \item $\PSL^{\epsilon}_n(q)$ with   $n\in \{3,4\}$ and $3|(q-\epsilon)$
    \item $\PSL^{\epsilon}_n(q)$ with $n\in \{4,5,6,7\}$ and $3|(q+\epsilon)$
    \item $\PSp_{2n}(q)$  with $n=2$ or $n=3$
    \item $\operatorname{P}\Omega_{2n+1}(q)$  with $n=2$ or $n=3$
     \item $\operatorname{P\Omega}_8^-(q)$
    \item $\operatorname{G}_2(q),\,^3\operatorname{D}_4(q)$ or $\,^2\operatorname{F}_4(q^2)'$(including  $\,^2\operatorname{F}(2)')$ 
    \item one of the alternating groups $A_n$ for $6\le n \le 11 $
    \item $\operatorname{M11},\operatorname{M12},\operatorname{M22},\operatorname{M23},\operatorname{M24}$
    \item $\operatorname{J2},\operatorname{J3},\operatorname{J4}$
    \item $\operatorname{Hs},\operatorname{Ru},\operatorname{He}$
\end{itemize}
\item[(II)] If $3$ divides $|A/S|$ then $A$ has $2$-generated Sylow $3$-subgroups if and only of $S$ is isomorphic to one of the following:
\begin{itemize}
    \item $\PSL_2(q)$ with $3\nmid q$
    \item $\PSL_3(q)$ with $3 |(q+\epsilon)$ or  $3|(q-\epsilon)$ and  $\langle F_0\rangle Q$ is not Sylow $3$-subgroup of $A$ for any $Q\in \Syl_3(S)$ and field automorphism $F_0$.
\end{itemize}
\end{enumerate}
\end{corollary}

    \section{Theorem A for groups of Lie type}
We have concluded our classification of  almost simple groups with socle a simple group of Lie type that have $2$-generated Sylow $3$-subgroups, and will now classify all almost simple groups $A$, with socle a simple group of Lie type,  such that $k_{0,\sigma}(B_0(A))\in \{6,9\}$. This discussion will rely heavily on Deligne–Lusztig theory. Important details of this theory will be discussed here and we refer the reader to \cite[Sec. 2.6]{GM20} for more detailed discussion. Let $G$ be a group of Lie type as defined in the opening paragraph of Section 3 and let $G^* = (\mathbf{G}^*)^F$, where $\mathbf{G^*}$ is in duality with $\mathbf{G}$ (see for example \cite[Def. 1.5.17]{GM20}). Let $\{s_1,...,s_k\}$ be a full set of conjugacy class representatives of semisimple elements of $G^*$. Then the irreducible characters of $G$ are partitioned into sets $\mathcal{E}(G,s_1),...,\mathcal{E}(G,s_k)$, which are referred to as rational Lusztig series. The elements of $\mathcal{E}(G,1)$ are referred to as unipotent characters, and the elements of $\mathcal{E}(G,s_i)$ are in bijection with the unipotent characters of $\mathbf{C}_{G^*}(s_i)$. If $G$ is defined over a field of characteristic $p$, then, for $\chi \in \mathcal{E}(G,s)$ corresponding to $\theta \in \mathcal{E}(\mathbf{C}_{G^*}(s),1)$, we have that $\chi(1) = [G^{*}:\mathbf{C}_{G^{*}}(s)]_{p'}\theta(1)$. If $i:\mathbf{G}\rightarrow \tilde{\mathbf{G}}$ is a regular embedding (see \cite[Section 1.7]{GM20}), $\tilde{G} = \tilde{\mathbf{G}}^F$ and $\chi \in \Irr(\tilde{G})$, we have that the number of irreducible constituents of $\chi_G$ divides $[\mathbf{C}_{\mathbf{G}^{*}}(s):\mathbf{C}^{\circ}_{\mathbf{G}^{*}}(s)]$, where $\mathbf{C}^{\circ}_{\mathbf{G}^{*}}(s)$ denotes the connected component of the identity in $\mathbf{C}_{\mathbf{G}^{*}}(s)$. We note that the principal block is contained in the union of all $\mathcal{E}(G,s)$  such that $s$ has $3$-power order by \cite[Theorem 9.12]{CE04}. Lastly we mention that, given a Galois automorphism $\sigma_0\in \mathrm{Gal}(\Q(\xi_{|G|})/\Q)$, where $\xi_{|G|}$ is a root of unity with $|\xi_{|G|}|=|G|$ and $\sigma_0(\xi_{|G|}) = \xi_{|G|}^r $, we have $\sigma(\mathcal{E}(G,s)) = \mathcal{E}(G,s^r)$ (see \cite[Proposition 3.3.15]{GM20}), so, in particular, if the conjugacy class of  $s^r \in G^*$ is not the same as that of $s$, then none of the elements of $\mathcal{E}(G,s)$ are $\sigma$-fixed.

We now begin our discussion with some useful lemmas.

\begin{lemma}\label{tilde}
Keep the setting above. Let $\ell\neq p$  be a non-defining prime dividing $|S|$ but not
dividing $|\tilde{S}/S|$. Then restriction provides a bijection between $\Irr_{0,\sigma}(B_0(\tilde{S}))$ and $\Irr_{0,\sigma}(B_0(S))$.
\end{lemma}    

\begin{proof}
     \cite[Lemma 5.1]{sf}  states that restriction defines a bijection between $\Irr(B_0(\tilde{S}))$ and $\Irr(B_0(S))$. This and \cite[Lemma 3.5]{mmsv} give the statement.
\end{proof}

    \begin{lemma}\label{lessthan9}
        Let $G$ be a finite group such that $P \in \Syl_3(G)$ is not cyclic and $k_0(B_0(G)) \le 9$. Then $k_{0,\sigma}(B_0(G)) \in \{6,9\}$
    \end{lemma}

    \begin{proof}
        By \cite[Lemma 1.4]{RSV20} we have that $3|k_{0,\sigma}(B_0(G))$, thus we have $k_{0,\sigma}(B_0(G)) \in \{0,3,6,9\}$. Further, since $\Irr_0(B_0(G))$ contains the principal character, we have $k_{0,\sigma}(B_0(G))$ cannot equal $0$, and, since $G$ does not have cyclic Sylow $3$-subgroup, \cite[Theorem A]{RSV20} implies $k_{0,\sigma}(B_0(G)) \neq 3$. Therefore, $k_{0,\sigma}(B_0(G)) \in \{6,9\}$.
    \end{proof}

We now finish the proof of Theorem A in the case where $S$ is a group of Lie Type defined in characteristic $3$.

\begin{proposition}\label{defchar}
    Let $S$ be a simple group of Lie type defined in characteristic $3$ and let $S\le A \le \Aut(S)$. Then
    Theorem A holds for $A$.
    \end{proposition}

    \begin{proof} Recall that throughout we exclude $\,^2F_2(2)'$, as it was considered in Proposition \ref{gapstuff}.
        From \cite[Lemma 4.2]{mmsv} we have that all characters in $\Irr(S)$ are $\sigma$-fixed. We may assume $S$ is not isomorphic to any of the groups discussed in Propositions \ref{An} or \ref{gapstuff}, so  \cite[Proposition 5.6]{sf} and Proposition \ref{def} give the desired result in the case of $A=S$. In the case where $A>S$ with $3\nmid|A/S|$, \cite[Proposition 5.6]{sf} and \cite[Lemma 3.5]{mmsv} give the desired result. 

        We now consider the case where $3$ divides $|A/S|$. The proof of \cite[Proposition 5.7]{sf} gives $12$ characters in $\Irr_{0}(B_0(A))$ above $4$  distinct $A$-orbits in $\Irr_{0}(B_0(S))$. Notice that, if $S\neq \operatorname{P\Omega}^+_8(q)$, $A/S$ has a normal $3$-complement since $3\nmid |\tilde{S}/S|$ and further note that in all of these cases $|A/S|$ has cyclic Sylow $3$-subgroups. Assume $S\neq \operatorname{P\Omega}^+_8(q)$, and let $M$ be the preimage of the $3$-complement under the canonical projection map $\pi :A \rightarrow A/S$. Let $\chi$ be one of the $12$ characters in $\Irr_{0,\sigma}(B_0(A))$ discussed above and let $\theta $ be an irreducible constituent of $\chi_M$. Then \cite[Proposition 4.2 and Lemma 3.5]{mmsv} give that $\theta $ is $\sigma$-fixed. Further, since $3\nmid\chi(1)$ and $A/M$ has $3$-power order, Clifford theory gives that $\theta$ must be $A$-invariant. Since $A/M$ is cyclic, this gives that $\theta $ extends to $A$. This extension must be in the principal block since $A/M$ has $3$-power order. Since $S \le M'$, we have $3\nmid|M/M'|$ and  the determinental order of $\theta$ must be prime to $3$.  We then have by \cite[corollary 6.4]{Nav18} that $\theta$ has a $\sigma$-fixed extension. Then, applying Gallagher's Theorem, we obtain $3$ $\sigma$-fixed extensions of $\theta$ corresponding to  $3$ $\sigma$-fixed linear characters of $A/M$, which are inflations of the characters of the unique factor of $A/M$ of order $3$. Since $|A/M|$ is  a power of $3$, all of these characters must lie in the principal block. We can repeat this argument for characters above each of the $4$ distinct $A$-orbits mentioned above to obtain $k_{0,\sigma} (B_0(A))\ge 12$. 

        We now consider $S = \operatorname{P\Omega}^+_8(q)$. First note that, since $3\nmid |\tilde{S}/S|$ and $\Aut(S)/ S\cong \tilde{S}/S\rtimes  (S_3\times C_k)$ for some $k$, $\Aut(S)$ has an index-$2$ subgroup $N_0$ which contains $S$ such that $N_0/S$ has a normal $3$-complement. It follows $A/S$ either has a normal $3$-complement or $A$ has an index $2$ subgroup $N := A\cap N_0$ such that $ N/S$ has a normal $3$-complement in $A$. In the first case we may argue as above, so assume the second case. First, we claim that there are at least $6$ $X$-invariant characters in $\Irr_{0}(B_0(S))$, where $X \in \Syl_3(A/S)$. The degrees of characters with multiplicities for $\tilde{S}$ can be found on  \cite{lubeckwebsite}. From this we obtain 6 character degrees which are $3'$ with $3'$ multiplicity in $\tilde{S}$. This gives $6$ $X$-invariant characters in $\tilde{S}$. Let $\theta$ be one such character.
        We then see that, since $3\nmid|\tilde{S}/S|$,  at least one of the constituents of $\theta_S$ must be $X$-invariant. This shows the claim. 
        Now we can argue as in the previous case to obtain $k_{0,\sigma}(B_0(N))\ge 18$. This and Lemma \ref{3prime} give that it is sufficient to show that these characters lie in at least $10$ distinct $A$-orbits. The orbits having size at most $2$ and the fact that the principal character lies in an orbit of size $1$ are sufficient for the desired $10$ orbits. 
    \end{proof}

      We now shift our attention to the case where $S$ is defined in characteristic $p \neq 3$.
      \begin{lemma}\label{4unipotents}
      Let S be a simple group of Lie type defined over $\mathbb{F}_q$, where $q$ is not a power of 3. Further, assume $3$ divides $|S|$. Then $\Irr_{0}(B_0(\tilde S))$ contains at least 4 rational valued unipotent characters that extend to $\Aut(S)$ and restrict irreducibly to $S$ unless $S$ is one of the following:

      \begin{itemize}
          \item $S = \PSL^\epsilon_4(q)$ with $3|(q -\epsilon)$, $S= \PSp_4(2^a)$ or $\PSL_3(q)$ with $3|(q-\epsilon)$, in which case $S$ has 3 such characters;
          \item $S= \PSL^\epsilon_3(q)$ with $3|(q+\epsilon)$ or $\PSL_2(q)$, in which case $S$ has $2$ such characters.
      \end{itemize}

  \end{lemma}

  \begin{proof}
      The appropriate characters can be chosen as in \cite[3.9-3.11]{RSV20} . For the classical groups all unipotent characters in $\tilde{S}$ are rational, and in all exceptional groups principal series, trivial and Steinberg characters were chosen except in the case of $G_2(q)$, where $3$ such characters were chosen. For the remaining character of  $G_2(q)$ we have \cite[Table 1]{Gec03} gives the chosen character is rational.
  \end{proof}

\begin{proposition}\label{theo:deg}
Let $S$ be one of the following simple groups
\begin{enumerate}
    \item $\operatorname{\PSL}^{\epsilon}_n(q)$ with $3|(q-\epsilon)$ and $n\ge5$
    \item $\operatorname{\PSL}^{\epsilon}_n(q)$ with $3|(q+\epsilon)$ and $n\ge 8$
    \item $\operatorname{\PSp}_{2n}(q)$ with $n\ge 4$
    \item $\operatorname{P\Omega}_{2n+1}(q)$ with $n\ge 4 $
    \item $\operatorname{P\Omega}^{-}_{2n}(q)$ with $n\ge 5$
    \item $\operatorname{P\Omega}^{+}_{2n}(q)$ with $n\ge 4$

\end{enumerate} and let $S\le A\le \Aut(S)$ with $3 \nmid |A/S|$. Then Theorem A holds for $A$.

\end{proposition}

\begin{proof} By Propositions \ref{theo:classical} and \ref{theo:psl} it is sufficient to show that $k_{0,\sigma}(B_0(A))\ge 10 $ in all cases.  For the majority of the proof we will proceed by choosing characters similarly to how they were chosen in \cite[Theorem 5.23]{sf} and argue as to which of these characters are $\sigma$-fixed. We first consider the case where $S = \PSL^{\epsilon}_n(q)$. Let $G = \SL^{\epsilon}_n(q)$ and $\tilde{G} = \GL^{\epsilon}_n(q)$. 
First, let $n =5 $, let $q$ be a power of a prime such that $3|(q-\epsilon)$, and let  $\omega$ denote a primitive third root of unity in $\mathbb{F}_{q^2}$. Using \cite{FS82} we see that the principal block is the unique unipotent block of $S$ and contains 6 unipotent characters of $3'$ degree. Thus, we see that any deflations of characters in series corresponding to semisimple $3$-elements in  $G^*$ must lie in the principal block by \cite[Theorem 9.12 and Lemma 17.2]{CE04}. The unipotent characters are $\sigma$-fixed by \cite[Lemma 4.7]{mmsv} and are in distinct $\Aut(S)$ orbits by \cite[Theorem 2.5]{malle08}. This and Lemma \ref{3prime} give that it is sufficient to find $4$ additional orbits of characters in $\Irr_{0,\sigma}(B_0(S))$.  Let $s,t \in \tilde{G}$ be the semisimple elements    $\mathrm{diag}(\omega,\omega,\omega,1,1)$  and $\mathrm{diag}(\omega,\omega,\omega,\omega,\omega^{-1})$ respectively.  Again using the work of \cite{FS82}, we obtain the structure of the centralizers of $s$ and $t$ and consequently character degrees of the characters in $\mathcal{E}(\tilde{G},s)$ and in $\mathcal{E}(\tilde{G},t)$. We see that the former has $2$ characters of distinct $3'$ degrees and the latter has $3$. Thus, by \cite[Proposition 3.3.15]{GM20} the characters in these series are $\sigma$-fixed. We have further that $(|s|,|Z(\mathbf{G})|) = (3,(n,q-\epsilon)) = 1$ and $\mathbf{C}_{\mathbf{G}^*}(s)$ is connected by \cite[Proposition 14.20]{MT11}, thus each element of $\mathcal{E}(\tilde{G},s)$ and of $\mathcal{E}(\tilde{G},t)$ has a unique irreducible constituent in $G$. Note that for any $ z \in \mathbf{Z}(\tilde{G})$ we can argue by dimension of eigenspaces that $s$ is not conjugate to $zt$, which gives that the characters below elements of $\mathcal{E}(\tilde{G},s)$ are distinct from those below elements of $\mathcal{E}(\tilde{G},t)$. These constituents must therefore be $\sigma$-fixed and lie in the principal block. If $\pi:\tilde{G}\rightarrow G^{*}$ is the canonical projection map, from the structure descriptions given in Proposition \ref{theo:psl} we see that $\pi(s)$ and $\pi(t)$ centralize Sylow $3$-subgroups of $G^*$.  Since $s$ and 
$t$ are also in $[G^*,G^*]$, \cite[Lemma 4.4]{NT13} gives that the characters in $\mathcal{E}(G,\pi(s))$ and  $\mathcal{E}(G,\pi(t))$ are trivial on $\mathbf{Z}(G)$ and we may take their deflations, which lie in the principal block by \cite[Lemma 17.2]{CE04}. Thus, we have more than $10$ $\Aut(S)$-orbits of characters in $\Irr_{0,\sigma}(B_0(S))$. 

Now assume $n \neq 5$. 
    By \cite[Proposition 5.22]{sf} we have that, in all cases, $S$ has at least 9 unipotent $3'$-degree characters in distinct $\operatorname{Aut(S)}$ orbits in its principal block. The fact that the unipotent characters are rational, the fact that non-unipotent characters cannot be conjugate to unipotent characters, and 
 Lemma \ref{3prime} give that it is sufficient to find one non-unipotent character in $\Irr_{0,\sigma}(B_0(S))$. From the proof of \cite[Theorem 3.4]{GSV} we see that a sufficient condition for the semisimple character $\chi_s \in \Irr(G)$ to contain $\mathbf{Z}(G)$ in its kernel is that $s \in \mathbf{Z}(Q) \cap (G^*)'$, where $Q \in \syl_3(G^*)$. The same proof in loc. cit. gives that this intersection is non-trivial.  If $3\nmid n$ or $3|(q+\epsilon)$, we  take any $s\in \mathbf{Z}(Q) \cap (G^*)'$ of order $3$ and we have $(|s|,|Z(\mathbf{G})|) = (3,(n,q-\epsilon)) = 1$, where $\mathbf{G}$ is as in the introduction to Section $3$, and $C_{\mathbf{G}^*}(s)$ is connected by \cite[Proposition 14.20]{MT11}. Thus, there is a unique semisimple character in $\mathcal{E}(G,s)$. Thus, $\chi_s$ is the only character of its particular degree in $\mathcal{E}(G,s)$ and  \cite[Proposition 3.3.15]{GM20} implies that $\chi_s$, and consequently its deflation are $\sigma$-fixed characters of $3'$-degree. We also see that \cite[Lemma 4.6]{mmsv} implies that $\chi_s$ lies in the principal block and \cite[Lemma 17.2]{CE04} gives that its deflation does as well.

     Thus, we are left with the case where $3|n$ and $3|(q-\epsilon)$. If $n$ is not a power of 3, then arguing as in the proof of Proposition \ref{theo:gl} we have that $P \in \syl_3(\tilde{G})$ can be written in a block diagonal form with multiple blocks. Since $n>2$ at least one of these blocks will have size that is a multiple of 3. Suppose  the block form is as follows:
    $$
      \begin{pmatrix}
A_1&0& ...\\ 
0&A_2& ...\\
\vdots &\vdots&\ddots\\
\end{pmatrix}
$$ were the $A_1$ is a $3k\times3k$ matrix. Then the semisimple element of $\GL^{\epsilon}_n(q)$ given by $s = \mathrm{diag}(\omega,...,\omega,1,...,1)$, where $\omega$ is a primitive third root of unity and has multiplicity $3k$, clearly centralizes $Q$. Also, we have $\mathrm{det}(s) = 1$, so the image under the projection map $\pi$ into $\operatorname{PGL}^{\epsilon}_n(q)$ must be in the commutator subgroup. Thus, by the discussion in \cite{GSV}, we have that the semisimple character corresponding to $\pi(s)$ has $\mathbf{Z}(G)$ in its kernel and we may take its deflation. Arguing as in the proceeding paragraph, it is sufficient to show that $C_{\mathbf{G^*}}(\pi(s))$ is connected. From  \cite[Remark 2.6.15]{GM20} we see that a sufficient condition for $C_{\mathbf{G^*}}(\pi(s))$ to be connected is that for all non-trivial $z\in \mathbf{Z}(\mathbf{\tilde{G}})$ we have $zs$ is not conjugate to $s$ in $\mathbf{\tilde{G}}$, where $\mathbf{\tilde{G}}^{F} = \tilde{G}$ with $F$  a Steinberg map as in the introduction to Section 3. Assume there exists a non-trivial $ z\in \mathbf{Z}(\mathbf{\tilde{G}})$ such that $zs$ is conjugate to $s$. Then there exists $a \in \bar{\mathbb{F}}_p$ such that $a\omega = 1$ and $a\cdot 1= \omega$, which is absurd. We can also argue as above to see that these characters lie in the principal block.

Lastly we consider the case in which 
 $n = 3^i$ with $i >1$. We then choose $s = \mathrm{diag}(\omega,...,\omega,\omega^{-1},...,\omega^{-1},1,...,1)$ where $\omega$ is a primitive third root of unity and each of $\omega,\omega^{-1}$ and $1$ have the same multiplicity. If $\pi:\tilde{G}\rightarrow G^{*}$ is the canonical projection map, we can argue as in the above cases to obtain that a semisimple character $\chi_{\pi(s)}$ corresponding to $\pi(s)$ is an element of $\Irr_0(B_0(G))$ and has $\mathbf{Z}(G)$ in its kernel. Thus, it is sufficient to show $\chi_{\pi(s)}$ is $\sigma$-fixed. First we consider the semisimple character  $\chi_s \in\mathcal{E}(\tilde{G},s)$. Since no other characters in $\mathcal{E}(\tilde{G},s)$ have the same degree as $\chi_s$ and since $\sigma$ permutes that characters of  $\mathcal{E}(\tilde{G},s)$  by \cite[Proposition 3.3.15]{GM20}, we see that $\chi_s$ is $\sigma$-fixed. Since $\chi_{\pi(s)}$ is a constituent of $(\chi_s)_G$ we have that  \cite[Proposition 3.3]{fieldofvalues} gives $\mathbb{Q}((\chi_s)_G) = \mathbb{Q}(\chi_{\pi(s)})$ and $\chi_{\pi(s)}$ must be $\sigma$-fixed as desired.

    Now  consider the case of $S = \operatorname{P\Omega}_8^{+}(q)$. Using   \cite{FS89} we see that in the case where $3|(q-1)$ the principal block is the unique unipotent block of $\operatorname{SO}^+_8(q)$  and we therefore have that all characters in series corresponding to semisimple elements of $3$-power order lie in the principal block. If instead $3|(q+1)$ we see that the principal block is a unique unipotent block of maximal defect, thus all $3'$-degree characters in series corresponding to semisimple of $3$-power order lie in the principal block by \cite[Theorem 9.12]{CE04}.  Considering semisimple elements in $\operatorname{(CO^\circ})^+(q)$ with eigenvalues  $(\omega,\omega, I_6)$ and $(\omega I_3,\omega^{-1}I_3,1,1)$ respectively and arguing as in the above case gives more than $9$ $\Aut(S)$ orbits in  $\Irr_{0}(B_0(S))$ of characters which are $\sigma$-fixed. This is sufficient by Lemma \ref{3prime}.

Now let $S$ be one of the remaining cases. 
    Then, \cite[Proposition 5.22]{sf} gives that, in all cases, $S$ has at least 9 unipotent $3'$-degree characters in distinct $\operatorname{Aut(S)}$ orbits in its principal block. Thus, we can argue as in the preceding cases that it is sufficient to find one non-unipotent character in $\Irr_{0,\sigma}(B_0(S))$. From \cite[Theorem 3.4]{GSV} we have that there exists $\chi \in \Irr_{0}(S)$ that is the deflation of a semisimple character $\chi_s \in \Irr(G)$ for a semisimple 3-element of $G^*$. We see further from the proofs of loc. cit. that, in the cases that $\mathbf{G}$ of type B,C, or D, where $S = 
 \mathbf{G}^F/\mathbf{Z}(\mathbf{G}^F)$, we may take our semisimple element to be any semisimple element in the center of a Sylow 3-subgroup, thus we may choose one of order 3. Note that in all of these cases we additionally have that $(|s|,|\mathbf{Z}(\mathbf{G})|) = 1$, thus $C_{\mathbf{G}^*}(s)$ is connected by \cite[Proposition 14.20]{MT11}. Thus, we can argue as in previous cases that the deflation of $\chi_s$ in $S$ lies in $\Irr_{0,\sigma}(B_0(s))$.
\end{proof}

\begin{proposition}\label{ad}
    Let $S\le A \le \Aut(S)$ with $S$ being one of the following:
    \begin{enumerate}
    \item $\PSL_2(q)$
        \item $\operatorname{\PSL}^{\epsilon}_n(q)$ with $3|(q+\epsilon)$ and $3\le n\le7$
    \item $\operatorname{\PSp}_{2n}(q)$ with $n= \{2,3\}$
    \item $\operatorname{P\Omega}_7(q)$
    \item $\operatorname{P\Omega}^{-}_8(q)$ \item$\operatorname{G}_2(q),\,^3\operatorname{D}_4(q)$ or $\,^2\operatorname{F}_4(q^2)'$
    \end{enumerate}
    
    Further, assume $3\nmid |A/S|$. Then $\tilde{S}A = \tilde{S}C_{A}(P)$, where $P \in \Syl_3(\tilde{S})$
\end{proposition}

\begin{proof}
Note that (6) was shown in \cite[Proposition 5.17]{sf}, so we need only consider (1) through (5). A special case of the remaining  statements was considered in loc. cit., but we include a full proof for completeness.
    Using \cite[Theorem. 4.10.2]{GLS} we see that $P$ is of the form $P_T \rtimes P_W$, where $P_T$ is the Sylow $3$-subgroup of a torus $T\le \tilde{S}$ and $P_W$ is isomorphic to an element of $ \Syl_3(N_{\tilde{S}}(T)/T)$. We also have that we can take coset representatives of  $A/(\tilde{S}\cap A)$  that are either a field automorphism, a graph automorphisms or the product of a field automorphism with a graph automorphism. Therefore, it is sufficient to find $g \in \tilde{S}$ such that $gF_0\in C_A(P)$ and $g\tau \in C_A(P)$, where $F_0$ is a generator of the cyclic groups of field automorphisms of order prime to $3$ and $\tau$ is a possibly trivial graph automorphism. Let $q = p^{m3^a}$, where $3 \nmid m $. Since $3\nmid |A/(\tilde S\cap A)|$, we have that any field automorphism coset representative of  $|A/S|$ must be of the form $F_p^{d3^a}$, where $d | m$. We have that $F_p^{d3^a}$ commutes with $P_W$. We also have that $P_T$ is the direct product of cyclic groups of $3$-power order and further that $F_p^{d3^a}$ acts coordinate by coordinate. Since $|\Aut(C_{3^b})| = 2\cdot 3^{b-1}$ and $3 \nmid|F_p^{d3^a}|$, we see that $F_p^{d3^a}$ either acts trivially on $P_T$ or by inversion. 
    
    We first consider the case where  $3|(q-1)$ and $S \neq \PSL^{-1}_n(q)$.  In this case we see that, if they act non-trivially, both $F$ and $\tau$ act by inversion on $P_T$ and commute with $P_W$, so it is then sufficient to find $g \in \tilde{S}$ that acts as inversion on $P_T$ and commutes with $P_W$. Considering the structure descriptions given in \cite{weir} we have that, given appropriate choice of basis, the images of 

    $$
    \begin{pmatrix}
0&1\\ 
1&0
\end{pmatrix},
   \begin{pmatrix}
0&I_n\\ 
I_n&0
\end{pmatrix},
    \begin{pmatrix}
0&I_3& 0\\ 
I_3&0& 0\\
0&0&-1\\
\end{pmatrix},
\text{and}
\begin{pmatrix}
0&I_3& 0&0\\ 
-I_3&0&0&0\\
0&0&-1&0\\
0&0&0&-1
\end{pmatrix}
    $$
    under projection map $\pi:G \rightarrow G/\mathbf{Z}(G) = S$ suffice for (1),(3),(4)
and (5) respectively.

We now consider all remaining cases. We have that $P_T$ is of a block diagonal form $\mathrm{diag}(A_1,...,A_k)$, where each $A_i$ is a $2\times 2$ matrix with eigenvalues $\zeta,\zeta^{-1}$ with $|\zeta|= 3^b$ for some $b$. A straightforward calculation gives that the rational canonical form of such a block is 
$
\begin{pmatrix}
    0&-1\\
    1&\zeta+\zeta^{-1}
\end{pmatrix}
$. 
Thus, if we chose a basis in which the standard generators of the $P_T$ are in rational canonical form, it is easy to see that $F$ cannot act by inversion and must act as identity. As for the situations in which $A$ induces a graph automorphism $\tau$ of order 2, which acts as inverse transpose, we can explicitly find $g\in \tilde{S}$ such that  $g\tau\in C_A(P)$. 
\end{proof}

\begin{lemma}\label{index2}
    Let $S = \PSL^{\epsilon}_n(q)$, $P\in \Syl_3(S)$ and $S\le A \le \Aut(S)$ such that $3\nmid |A/S|$. Then $[A: A\cap \tilde S C_A(P)]\le 2$. 
\end{lemma}
\begin{proof}
    We have that $A/(\tilde S\cap A)$ is generated by field automorphisms and graph automorphisms. Arguing as in the above proposition, we have that, for a field automorphism $F_0$,  either $F_0$ or $F_0^2$ is an element of $C_A(P)$. If $F_0 \in C_A(P)$, then we have that $A\cap \tilde SC_A(P)$ must have at most one nontrivial coset with representative a graph automorphism $\tau$. If $F_0 \notin C_A(P)$ and $\tau \in A$, then $F_0\tau\in A\cap \tilde SC_A(P)$, meaning they lie in the same coset and again  $A\cap \tilde SC_A(P)$ must have only one nontrivial coset.
\end{proof}

\begin{proposition}\label{exceptional}
    Let $S$ be an exceptional simple group of Lie type and let $S\le A \le \Aut(S)$ such that $3\nmid |A/S|$. Then Theorem A holds for $A$.
\end{proposition}

    \begin{proof}
Recall that $3\nmid |\,^2\operatorname{B}_2(q^2)| $  and $\,^2\operatorname{G}_2(q^2)$ is only defined in characteristic $3$ and was considered in Proposition \ref{defchar}, so we need not consider them here. If $S$ is one of $\operatorname{F}_4(q),\operatorname{E}_6^{\epsilon}(q),\operatorname{E}_7(q)$, or $\operatorname{E}_8(q)$, Proposition \ref{theo:classical} gives that $A$ has Sylow $3$-subgroups which are not $2$-generated. Thus, it is sufficient to show $k_{0,\sigma}(B_0(A))\ge 10$. The proof of \cite[Proposition 5.21]{sf} exhibits $10$ unipotent characters in $k_{0}(B_0(S))$ in distinct $\Aut(S)$-orbits. Further, these characters are $\sigma$-fixed by \cite[Lemma 4.7]{mmsv}. This and Lemma \ref{3prime} together give the desired result.

 Now let $S$ be one of $\operatorname{G}_2(q),\,^4\operatorname{D}_3(q)$ or  $\,^2\operatorname{F}_4(q^2)'$. We have that \cite{HiSh90}, \cite{DM87} and \cite{Mal90} respectively give $k_0(B_0(S))\le 9$. Further, Proposition \ref{theo:classical} gives that these groups have $2$-generated Sylow subgroups. Therefore, Lemma \ref{lessthan9} gives the result in the case of $S=A$. If we instead have $A > S$, then Proposition \ref{ad}, Lemma \ref{tilde} and Theorem \ref{ald} give the desired result. 
    \end{proof}

\begin{proposition}\label{psl3}
    Let $S = \PSL^{\epsilon}_3(q)$ with $3|(q-\epsilon)$ and let $S\le A \le \Aut(S) $. Then Theorem A holds for $A$ 
\end{proposition}

\begin{proof}
      From Proposition \ref{4unipotents} we obtain $3$ unipotent characters in $k_0(B_0(S))$ and further from the generic character table found in \cite{SF73} we see that there are $3$ additional characters in $k_0(B_0(S))$. Thus, \cite[Theorem A]{RSV20} implies $k_{0,\sigma}(B_0(S)) = 6$, since the Sylow $3$-subgroups of $S$ are not cyclic by Proposition \ref{theo:psl}. Thus, the result holds in the case of $A=S$ by proposition \ref{2or3}.

      In order to show the case where $A>S$ we will first discuss the $\Aut(S)$-action on $\Irr_{0,\sigma}(B_0(S))$. Since the 3 unipotent characters extend to $\Aut(S)$, they must be invariant under the action $\Aut(S)$. Using \cite[Proposition 5.12]{mal20}  and \cite[Proposition 2.13]{MO83}  we obtain $k_0(B_0(\tilde S)) = 9$. Gallagher's Theorem and the fact that $|\tilde{S}/S|=3$ yields $9$ characters above the $3$ unipotent characters. Thus, there cannot be any $3'$-degree characters in $\Irr_0(B_0(\tilde S))$ that lie above the non-unipotent characters. It follows that the $3$ remaining characters must be interchanged by the action of the outer diagonal automorphisms. Now, in order to consider the action of field and graph automorphisms, let $\hat{S} = S \rtimes \langle \tau,F_p\rangle$, where $\tau$ is a graph automorphism in the case of $S=\PSL_3(q)$, $\tau$ is trivial in the case of $S=\operatorname{PSU}_3(q)$, and $F_p$ is a generator of the subgroup of field automorphisms. We can argue as in the proof of Lemma \ref{index2} to see that $[\hat{S}:SC_{\hat{S}}(P)] \le 2$, where the non-trivial coset can be taken to act by inversion on $P_T$. Then, using Theorem \ref{ald}, we obtain that restriction defines a bijection between $\Irr_{0}(SC_{\hat{S}}(P))$ and $\Irr_{0}(S)$. This implies that the inertial group in $\hat{S}$ of the non-unipotent characters in $\Irr_{0}(B_0(S))$ has index at most $2$ in $\hat S$ and we see further from the character table in \cite{SF73} that in the case that the index equals $2$ the $\hat{S}$ action interchanges $2$ of the non-unipotent characters.

     We first consider the case where $3\nmid |A/S|$. If $SC_A(P) = A$, then Theorem \ref{ald}, Proposition \ref{2or3} and 
    Lemma \ref{lessthan9} give the desired result and further that restriction defines a bijection between $k_{0,\sigma}(B_0(S))$ and $k_{0,\sigma}(B_0(A))$. Assume $[A:SC_A(P)] = 2$. Using \cite[Lemma 2.1]{sf} we see that the principal block of $A$ is the unique block lying above the principal block of $SC_A(P)$. As discussed above we see that $4$ of the characters in $\Irr_{0,\sigma}(B_0(SC_A(P)))$ are $A$-invariant and $2$  are interchanged. We then have that the $4$ fixed characters extend to $A$ and Gallagher's Theorem gives $8$ characters lying above them. Clifford correspondence gives one character lying above the $2$ interchanged characters. Thus, $\Irr_{0}(B_0(A)) = 9$ and Lemma \ref{lessthan9} give the result.

We now consider the case in which $3$ divides $|A/S|$. First assume $A/S$ has $2$-generated Sylow $3$-subgroups. In this case we see that, since $|\tilde{S}/S|=3$ and $A/(A\cap\tilde{S})$ is abelian, $A/S$ must have a normal Sylow $3$-subgroup. Let $X$ be the preimage of this Sylow $3$-subgroup of $A/S$ under the canonical projection map and let  $\theta_1,\theta_2,\theta_3$ be the characters described in Proposition \ref{4unipotents} restricted to $S$. Since each $\theta_i$ extends to $\Aut(S)$, we in particular have that each $\theta_i$ extends to $X$. Since $X/S$ is a $2$-generated $3$-group, we see it has at least $9$ $\sigma$-fixed characters corresponding to the inflations of the characters of $(X/S)/\Phi(X/S)$. We can apply \cite[Corollary 6.4]{Nav18} to obtain a
 $\sigma$-fixed extension of $\theta_i$ in $\Irr(X)$ and then apply Gallagher's Theorem to obtain such  $9$ $\sigma$-fixed extensions. Since the $\theta_i$ are in distinct orbits, this gives $k_{0,\sigma}(B_0(X))\ge27$. Then, if $XC_A(P) = A$,  Theorem \ref{ald} alongside Proposition \ref{2or3} gives the desired result. If instead we have  $[A:XC_A(P)] = 2$ (this is the only other possibility since $SC_A(P)\le XC_A(P)$ and $[A:SC_A(P)] \le 2$ by the above), Theorem \ref{ald} gives that $k_{0,\sigma}(B_0(XC_A(P))) \ge 27$. Using \cite[Lemma 2.1]{sf} we see that the principal block of $A$ is the unique block lying above the principal block of $SC_A(X)$. Each character in $\Irr_{0,\sigma}(B_0(A))$ lies above at most $2$ characters in $\Irr_{0,\sigma}(B_0(SC_A(P)))$. This and \cite[Lemma 3.5]{mmsv} imply $k_{0,\sigma}(B_0(A))  > 13$. This and Proposition \ref{2or3} give the result.

 Now Let $A > S$ such that the Sylow $3$-subgroups of $A/S$ are  cyclic and generated by a field automorphism. In this case we note that $A/S$ has a normal $3$-complement $M$. We may argue as in the first case to obtain $k_{0,\sigma}(B_0(M)) \in \{6,9\}$. In either case we obtain at least $4$ $A$-invariant characters. Arguing as in   Proposition \ref{allbutpsl3} we obtain $k_{0,\sigma}(B_0(A)) \ge 12 $.

Next we consider the case in which $A > S$ such that the Sylow $3$-subgroups of $A/S$ are cyclic and generated by an outer diagonal automorphism. From \cite[Proposition 5.19]{sf} we have $k_{0}(B_0(A)) = 9$ and by Proposition \ref{2or3} and Lemma \ref{lessthan9} we have that Theorem A holds. 

Lastly we consider the case in which $A > S$ such that the Sylow $3$-subgroups of $A/S$ are cyclic and generated by $\delta F_0$ where $\delta$ is an  outer diagonal automorphism and $F_0$ is a field automorphism for $3$-power order. We see that the action of $\langle\delta F_0\rangle$ fuses the $3$ non-unipotent characters and, arguing as above, we obtain $9$ characters lying above the $3$ unipotent characters in $S\langle\delta F_0\rangle$. Arguing as in the proof of Lemma \ref{index2} we obtain $[A:S\langle\delta F_0\rangle C_A(P)]\le 2$; however, if $[A:
S\langle\delta F_0\rangle C_A(P)]= 2$ we see that the representative for the nontrivial coset will invert $\delta$ and commute with $F_0$. Thus, we have $\delta^{-1}F_0 \in A$ and $\delta =\delta^{-1}F_0F_0^{-1}\delta^{-1} \in A$. This gives $A/S$ having $2$-generated Sylow subgroups which contradicts the assumption. Therefore, $A=S\langle\delta F_0\rangle C_A(P)$ and Theorem \ref{ald} gives the result.
\end{proof}

\begin{theorem}
    \label{theo:exceptions}
   If $S$ is a simple group of Lie type defined in characteristic apart from $3$ and  $S\le A \le \Aut(S)$ such that $3\nmid |A/S|$, then Theorem A 
 holds for $A$. \end{theorem}

\begin{proof}
From  Propositions \ref{exceptional} and \ref{theo:deg}  it is sufficient to consider the classical groups listed in the statement of  Proposition \ref{theo:classical}, $\PSL^\epsilon_3(q)$ with $3|(q+\epsilon)$, and $S= \PSL_2(q)$;  however, since the proof of Proposition \ref{theo:gl} gives that in the latter $2$ cases we have cyclic Sylow $3$-subgroups, \cite[Theorem A]{RSV20} gives the desired result in those cases. Therefore, we need only consider the classical groups listed in the statement of  Proposition \ref{theo:classical}.

We begin by compiling some result which will be useful throughout the proof.
The work of \cite{FS82} and \cite{FS89} describe which unipotent characters belong to the principal block as well as the structures of the centralizers of semisimple elements,  \cite[Theorem 2.5]{malle08} describes the $\operatorname{Aut}(S)$-action on the unipotent characters and \cite[Proposition 7.2]{Tay18} gives  a description of the $\Aut(S)$-action on the set of Lusztig Series. Also note \cite[Proposition 3.3.15]{GM20}, which describes the Galois action on the set of Lusztig series.
 
 We proceed by  choosing characters similarly to how they were chosen in \cite[Theorem 5.23]{sf} and argue as to which of these characters are $\sigma$-fixed.
We first consider cases in which $S=\PSL_4^{\epsilon}(q)$ with $3|(q-\epsilon)$. For ease of notation let $G = \SL^{\epsilon}_4(q)$ and $\tilde{G} = \GL_4^{\epsilon}(q)$. From Lemma \ref{tilde} we have that restriction defined a bijection between  $\Irr_{0,\sigma}(B_0(\tilde{S})) $ and $\Irr_{0,\sigma}(B_0(S) )$. Using \cite[Proposition 5.12]{mal20}  and \cite[Proposition 2.13]{MO83}  we obtain $k_0(B_0( S)) = k_0(B_0(\tilde S) )= 3^{a+1} $, where $(q-\epsilon)_3 = 3^a$. Consider a semisimple element $\tilde{s} \in\tilde{G}$ which has eigenvalues of the form $(\zeta,\zeta,\zeta,\zeta^{-3})$, where $\zeta = |3^b|$ with $0 < b \le a$. Any such $\tilde{s}$  will have centralizer of the form $\GL_3^{\epsilon}(q)\times C_{q- \epsilon}$, which has $3$ unipotent characters of distinct $3'$-degrees. Thus, $\mathcal{E}(\tilde{G},\tilde{s})$ will  contain three $3'$-degree characters, and these characters are $\sigma$-fixed if and only if $\tilde{s}$ is by \cite[Proposition 3.3.15]{GM20}.  If $\pi:\tilde{G}\rightarrow \tilde{S}$ is the canonical projection map and $s= \pi(\tilde{s})$, we can argue as in the second to last paragraph of the proof of Proposition \ref{theo:deg} to see that  $\pi(\tilde{s})$ has a connected centralizer, thus each element of $\mathcal{E}(\tilde{G},s)$ has a unique irreducible constituent in $G$, which will again be $\sigma$-fixed if and only if $\tilde{s}$ is.  Furthermore, $G$ has a unique unipotent block, implying $\mathcal{E}(G,s) \subseteq \Irr(B_0(G))$. Using \cite[Lemma 4.4]{NT13} we see that each of these constituents is trivial on $\mathbf{Z}(G)$ and we may take their deflation in $\Irr(B_0(S))$. Thus, for each such $s$ we obtain $3$ characters in $\Irr_{0,\sigma}(B_0(S))$, which are $\sigma$-fixed if and only if $s$ is.   We have there are $3^a-1$ choices for $s$ (note that each choice of $\zeta$ gives a distinct semisimple element mod the center, since any central element with determinant $1$ cannot translate one to another). Thus, we obtain $3(3^a - 1)$ characters in $\Irr_{0}(B_0(S))$. These and the 3 unipotent characters give all $3^{a+1}$ $3'$-degree characters. From the above discussion we have that the only fixed characters are the unipotent characters and the characters in series corresponding to the semisimple elements whose preimages have eigenvalues $(\omega,\omega,\omega,1)$ or $(\omega^{-1},\omega^{-1},\omega^{-1},1)$, giving 9 total. This and Proposition \ref{theo:psl} give the desired result in the case of $S = A$. By Lemma \ref{index2} and its proof we see have that $[\tilde{S}A : \tilde{S}C_A(P)]\le 2 $, and further that, in the case of $[\tilde{S}A : \tilde{S}C_A(P)]= 2 $, the nontrivial coset can be taken to act by inversion on $P_T \in  \Syl_3(T)$, where $T\le S$ is a maximal torus. Using Lemma \ref{tilde} and Theorem \ref{ald} we get that restriction gives a bijection between $\Irr_{0,\sigma}(B_0(S))$ and $\Irr_{0,\sigma}(B_0(\tilde{S}C_A(P))$, which implies that restriction defines a bijection between  $\Irr_{0,\sigma}(B_0(S))$ and $\Irr_{0,\sigma}(B_0(A \cap(\tilde{S} C_A(P)))$. If $\tilde{S}A = \tilde{S}C_A(P)$, then  $A =A\cap \tilde SC_A(P) $  and the result follows from the above discussion. If   $[\tilde{S}A : \tilde{S}C_A(P)]= 2 $, then the diamond isomorphism theorem gives $[A : A\cap \tilde SC_A(P)]= 2 $. Using \cite[Lemma 2.1]{sf} we have that the principal block of $A$ is the unique block above the principal block of $A\cap \tilde SC_A(P)$. Then the $3$ unipotent characters extend to $A$ whereas the semisimple characters corresponding to semisimple elements whose preimages have eigenvalues $(\omega,\omega,\omega,1)$ and  $(\omega^{-1},\omega^{-1},\omega^{-1},1) $
are interchanged by the action of $A$. Thus, Clifford correspondence gives $3$ characters above the $6$ corresponding to $(\omega,\omega,\omega,1)$ and $(\omega^{-1},\omega^{-1},\omega^{-1},1)$. Then Gallagher's Theorem gives $6$ characters above the $3$ unipotent characters. This and \cite[Lemma 3.5]{mmsv} give $k_{0,\sigma}(B_0(A)) = 9$.

For all remaining cases we have that Lemma \ref{tilde}, Proposition \ref{ad} and Theorem \ref{ald} give that restriction defines a bijection between $\Irr_{0,\sigma}(B_0(S))$ and $\Irr_{0,\sigma}(B_0(A
))$. This alongside Proposition \ref{theo:classical} and Lemma \ref{lessthan9} give that it is sufficient to show that $k_{0,\sigma}(B_0(S)) \le 9$ in all remaining cases.

 Consider the case where $3|(q+\epsilon)$ and  $S$ is one of $\PSL_4^\epsilon(q)$ or $ \PSL_5^\epsilon(q)$. In both cases there are $5$ unipotent characters of $3'$-degree in the principal block, thus it is sufficient to argue that there are at most $4$ additional $\sigma$-fixed characters in $\Irr_0(B_0(S))$. First consider the   semisimple  elements $s,t \in \tilde{G}:= \GL^{\epsilon}_n(q)$ that have nontrivial eigenvalues $(\omega,\omega^{-1})$ and $(\omega,\omega,\omega^{-1},\omega^{-1})$ respectively. We see that  have $\mathbf{C}_{\tilde{G}}(s) \cong \Gl^\epsilon_1(q^2)\times\GL^\epsilon_2(q)$ and $\mathbf{C}_{\tilde G}(t) \cong \Gl^\epsilon_2(q^2)$ when $n = 4$ as well as  $\mathbf{C}_{\tilde{G}}(s) \cong \Gl^\epsilon_1(q^2)\times\GL^\epsilon_3(q)$ and $\mathbf{C}_{\tilde G}(t) \cong \Gl^\epsilon_2(q^2)\times\GL^\epsilon_1(q)$ when $n = 5$. In either case we have that both $\mathbf{C}_{\tilde{G}}(s)$ and $\mathbf{C}_{\tilde{G}}(t)$ have two $3'$-degree unipotent characters of distinct degrees. We can then argue as in the previous case to obtain $4$ characters in $\mathcal{E}(\SL_n(q),\pi(s))\cup \mathcal{E}(\SL_n(q),(t))$, where $\pi$ again denotes that canonical projection map from $\tilde{G}$ to $\tilde{S}$.  These and the unipotent characters give $9$ total. Thus, it suffices to show that all other characters in the principal block are not $\sigma$-fixed. 
 Using \cite[Theorem 9.12 and Lemma 17.2]{CE04} we see that it is sufficient to consider deflations of characters in series of the form $\mathcal{E}(\SL^{\epsilon}_n(q),r)$, where $r$ has $3$ power order.  The only possibilities that have not already been considered are those where $|r| \ge 9$; however, \cite[3.3.15]{GM20} gives that elements in these series cannot be $\sigma$-fixed, which gives the desired result.

Now consider the case were $3|(q+\epsilon)$ and  $S$ is one of $\PSL_6^\epsilon(q)$ or $S = \PSL_7^\epsilon(q)$. Let $G = \SL^{\epsilon}_n(q)$, $\tilde{G} = \GL^{\epsilon}_n(q)$, and $3^{a}= (q+\epsilon)_3$ in both cases.  Using Lemma \ref{tilde}
  we have that $k_0(B_0(S)) =  k_0(B_0(\tilde S))$ and the fact the $3\nmid|Z(\tilde{G})|$ gives $k_0(B_0(\tilde{S})) = k_0(B_0(\tilde{G}))$ by \cite[Theorem 9.9(c)]{Nav98}. Using \cite[Proposition 2.13]{MO83} and the preceding discussion we have that in both cases $k_0(B_0(S))  = 3(3^a-1)/2 + 6$. We consider semisimple elements  of $\tilde{G}$ which have non-trivial eigenvalues of the form $(\zeta,\zeta,\zeta,\zeta^{-1},\zeta^{-1},\zeta^{-1})$ with $|\zeta| = 3^b$ with $0 <b  \le a$. There are $(3^a - 1)/2$ such semisimple elements up to conjugacy. We then argue as in the case of $S= \PSL^{\epsilon}_4(q)$ above to obtain $3(3^a-1)/2$ characters in the principal block of $S$, exactly 3 of which (those corresponding to the semisimple element whose eigenvalues are third roots of unity) are $\sigma$-fixed. These and the $6$ unipotent characters of $3'$-degree in the principal block account for all $3'$-degree characters in the principal block and only the unipotent characters and the $3$ discussed above are $\sigma$-fixed, giving $9$ total.

 Now let $S = \PSp_4(q)$. We then have that the principal block is the unique unipotent block of maximal defect and contains $5$ unipotent characters of $3'$ degree. Arguing as in the case of $\PSL^{\epsilon}_5(q)$ above we see that it is sufficient to count characters corresponding to semisimple elements of order 3 in $G^* = \operatorname{SO}_5(q)$. The only possibilities are the semisimple elements with eigenvalues $(\omega,\omega,\omega^{-1}\omega^{-1},1)$ and $(\omega,\omega^{-1},1,1,1)$, both of which have centralizers with 2 unipotent characters of distinct $3'$-degrees. These and the unipotent characters are the only characters that are possibly $\sigma$-fixed, so $k_{0,\sigma}(B_0(S)) \le 9$ as desired.

 Now consider $S = \PSp_6(q)$. By \cite[Theorem 9.9(c)]{Nav98} we have $\Irr_{0}(B_0(S)) = \Irr_{0}(B_0(\operatorname{Sp}_6(q)))$, since $3 \nmid |\mathbf{Z}(\operatorname{Sp}_6(q))|$.  Therefore, using \cite[Theorem 5.17]{mal20},  we obtain $\Irr_0(B_0(S)) = 6 + 3\cdot(3^a-1)/2$, where $3^a = (q^2-1)_3$. Note that, arguing similarly as in the case of $S= \PSL_6(q)$, we find that  any semisimple element of $G^*=\operatorname{SO}_7(q)$ with eigenvalues  $(\zeta,\zeta,\zeta,\zeta^{-1},\zeta^{-1},\zeta^{-1},1)$, where $|\zeta|=3^k$ with $k 
 \ge 1$,  will yield $3$ characters in $\Irr_{0}(B_0(S))$. These and the $6$ unipotent character account for all characters in the principal block. Note that the only of these characters that could be $\sigma$-fixed are the unipotent characters and the 3 characters corresponding to the seimisimple element with eigenvalues $(\omega,\omega,\omega,\omega^{-1},\omega^{-1},\omega^{-1},1)$, giving $k_{0,\sigma}(B_0(S)) \le 9$ as desired. 

 For the case of $\operatorname{P\Omega}_7(q)$ we let $H = \operatorname{SO}_7(q)$.  We have that $S \le H/\mathbf{Z}(H) \le \tilde{S}$. Since $(|\mathbf{Z}(H)|,3) = 1 $ and $S \le H/\mathbf{Z}(H) \le \tilde{S}$ , we have by Lemma \ref{tilde} that $k_{0}(B_0(H)) =k_0(B_0(H/\mathbf{Z}(H)))  = k_0(B_0( S))$. Using \cite[Theorem 5,17]{mal20} as above to obtain $k_{0}(B_0(H))$ the result   follows similarly to the above case, but instead using semisimple elements with eigenvalues of the form $(\zeta,\zeta,\zeta,\zeta^{-1},\zeta^{-1},\zeta^{-1})$ in $\operatorname{Sp}_6(q)$, which is dual to $H$.

 Lastly we consider $S = \operatorname{P\Omega}_8^{-}(q)$. Let $H= \operatorname{SO}_8^-(q)$, $\tilde{H} = \operatorname{GO}_8^-(q)$ and $P \in  \Syl_3(H)$. We then calculate $\tilde{H} = HC_{\tilde{H}}(P)$. Applying Theorem \ref{ald} gives $k_{0}(B_0(H)) = k_{0}(B_0(\tilde{H}))$. We can then use \cite[Theorem 5.17]{mal20} to obtain $ k_{0}(B_0(H)) = 6 + 3(3^{a}-1)/2$. As in the preceding case we have $(|\mathbf{Z}(H)|,3) = 1 $ and $S \le H/\mathbf{Z}(H) \le \tilde{S}$. We also have by Lemma \ref{tilde} that $k_{0}(B_0(H)) =k_0(B_0(H/\mathbf{Z}(H)))  = k_0(B_0( S)) =6 + 3(3^{a}-1)/2$.  We consider each semisimple element with nontrivial eigenvalues of the form $(\zeta,\zeta,\zeta,\zeta^{-1},\zeta^{-1},\zeta^{-1})$ in $H$ and then argue as in the above cases we obtain $k_{0,\sigma}(B_0(S)) \le 9$ and the result.
 \end{proof}

\begin{proposition}\label{theo:psl2}
Let $S = \PSL_2(q)$ or let $3|(q+\epsilon)$ and let $S=\PSL_3^{\epsilon}(q)$. If $S \le A \le \Aut(S)$, then Theorem A holds for $A$. \end{proposition}

\begin{proof}

In both cases $S$ has cyclic Sylow $3$-subgroups. Therefore, if $3\nmid|A/S|$, the result follows from \cite[Theorem A]{RSV20}, so we may assume $3$ divides $|A/S|$. Since $3 \nmid (n,q-\epsilon)$, we can write $A $ in the form $M \rtimes \langle F_{q_0}\rangle$, where $S \le M \le A $ such that $3 \nmid  |M/S|$, and  $|F_{q_0}| = 3^k$ for some $k$. Using Proposition \ref{ad}, Lemma \ref{tilde} and Theorem \ref{ald} we see that restriction gives a bijection between  $\Irr_{0,\sigma}(B_0(M))$ and $\Irr_{0,\sigma}(B_0(S))$. Further, we have exactly $3$ $\sigma$-fixed $3'$ degree characters in the principal block of $M$  by \cite[Theorem A]{RSV20}.
We have that $2$ of the $3$ characters in $\Irr_{0,\sigma}(B_0(S))$ are the trivial character and the Steinberg character, so they extend rationally to $\Aut(S)$. Thus, applying Gallagher's Theorem to the extensions of these characters in $\Irr(B_0(M))$ gives 6 $\sigma$-fixed characters in $B_0(A)$ lying above these characters corresponding to the inflations of the  three $\sigma$-fixed $3'$ degree characters of $A/M$ which we obtain from \cite[Theorem A]{RSV20}. Let $\chi$ be the element of $\Irr_{0,\sigma}(B_0(M))$ which is not the principal character nor the extension of the Steinberg character. Then we claim $\chi$ must be $A$-invariant. This is due to the fact the action of $A$ will permute $\Irr_{0,\sigma}(B_0(M))$ and the principal character is invariant, so the remaining characters must be as well since $\Irr_{0,\sigma}(B_0(M))$ is of size $3$. Thus, since $A/M$ is cyclic, we have that $\chi$ extends to $A$. Since $S = S' \le M'$, we have that $|M/M'|_3=1$. Thus, we have $\mathrm{gcd}(o(\chi)\chi(1),[A:M]) =  1$. Thus, we apply \cite[Corollary 6.2]{Nav18} to obtain an extension of $\hat\chi$, which by \cite[Corollary 6.4]{Nav18} must be $\sigma$-fixed.  Again, using Gallagher's Theorem as above, we then obtain  exactly $3$ $\sigma$-fixed characters lying above $\chi$. Therefore, we obtain $k_{0,\sigma}(B_0(A)) = 9$. This and Proposition \ref{2or3} give the result.
 \end{proof}

\begin{proposition}\label{allbutpsl3}
    Let $S$ be a simple group of Lie type defined in characteristic $p \neq 3$  such that $S\neq \PSL_2(q)$ and $S\neq \PSL^{\epsilon}_3(q)$. Let  $S\le A \le \Aut(S)$ such that $A/S$ has nontrivial Sylow 3-subgroups. Then Theorem A holds for $A$.
\end{proposition}

\begin{proof}
By Proposition \ref{alm} we have that it is sufficient to show that $k_{0,\sigma}(B_0(A))\ge 10$. 
    From  Lemma \ref{4unipotents} we have that $\Irr_{0}(B_0(\tilde{S}))$ contains at least $4$ unipotent characters that extend to $\Aut(S)$ and restrict irreducibly to $\Irr_{0}(B_0(S))$ unless $S = \PSL_4^\epsilon(q) $ with $3|(q - \epsilon)$ or $S = \PSp_4(2^a)$. Let $S$ be not one of those exceptions. First, we consider the case where  $3\nmid |A/(\tilde{S}\cap A)|$. Let $\theta_1,...,\theta_4$ be the $4$ unipotent characters chosen in  Lemma \ref{4unipotents} restricted to $A\cap \tilde{S}$. By \cite[Lemma 4.7]{mmsv}  we have that each $\theta_i$  is $\sigma$-fixed. Then, using Lemma \ref{3prime}, we find some $\hat\theta_i\in \Irr_{0,\sigma}(B_0(A)|\theta_i)$ for each $\theta_i$.  Then \cite[Lemma 2.7]{sf} gives that $3|k_{0}(B_0(A)|(\theta_i)_S)$. This,  the fact that $\sigma$ has $3$-power order, and the fact $\sigma$ permutes the elements of $\Irr_{0,\sigma}(B_0(A)|(\theta_i)_S)$ gives that $3|k_{0,\sigma}(B_0(A)|(\theta_i)_S)$. Therefore, since $\hat\theta_i \in \Irr_{0,\sigma}(B_0(A)|(\theta_i)_S)$, we have that there are at least 3 characters in  $\Irr_{0,\sigma}(B_0(A)|(\theta_i)_S)$.  Since $(\theta_1)_S,(\theta_2)_S,(\theta_3)_S$ and $(\theta_4)_S$ are in distinct $\Aut(S)$-orbits, this gives $k_{0,\sigma}(B_0(A)) \ge 12$.

    Now assume instead that $3$ divides  $|A/(\tilde{S}\cap A)|$. Let $\theta_1,\theta_2,\theta_3$ and $\theta_4$ be the $4$ unipotent characters in $B_0(S)$ chosen in  Lemma \ref{4unipotents} and let $\theta = \theta_i$ for some $i$. Let $P$ be a Sylow $3$-subgroup of $A/(A\cap \tilde{S})$. Since $\theta$ was chosen to be rational and $\Aut(S)$-invariant, by \cite[Corollary 6.6(a)]{Nav18} we can extend $\theta$ to some   $\hat\theta \in \Irr_{0,\sigma}(B_0(\tilde{S} P))$. Then Lemma \ref{3prime} gives a character in $\Irr_{0,\sigma}(B_0(A)|\hat \theta_{\tilde{S}P\cap A})$. This character must lie in $\Irr_{0}(B_0(A)|\theta_S)$. Then, arguing as in the preceding paragraph, we see that  $k_{0,\sigma}(B_0(A)|\theta_S)\ge 3$.  Applying this argument to all $4$ of the $\theta_i$ yields $k_{0,\sigma}(B_0(A))\ge 12$ as they were chosen in distinct $\Aut(S)$-orbits.

Now let $S=\PSL^{\epsilon}_4(q)$ with $3|(q-\epsilon)$ or let $S = \operatorname{PSp_4(2^a)}$. Then by Proposition \ref{4unipotents} we have $\Irr_{0}(B_0(\tilde S))$ contains at least $3$ rational valued unipotent characters that extend to $\Aut(S)$ and restrict irreducibly to $S$. Arguing as in the above cases this gives $k_{0,\sigma}(B_0(A)) \ge 9$, thus it is sufficient to find one additional $\sigma$-fixed character. It is sufficient to find semisimple element $s\in G^*$ such that the semisimple character $\chi_s$ is contained in $\Irr_{0,\sigma}(B_0(S))$. If $S = \PSL_4^{\epsilon}(q) $ with $3|(q-\epsilon)$, we take the semisimple element $s \in G^* = \operatorname{PGL}^{\epsilon}_4(q)$ such that $s$ is the image under the canonical projection map of a semisimple  element of  $\GL^{\epsilon}_4(q)$ with eigenvalues $(\omega,\omega,\omega,1)$, where $\omega$ is a primitive third root of unity. Arguing as in the third paragraph of the proof of Proposition \ref{theo:exceptions} gives $\chi_s \in \Irr_{0,\sigma}(B_0(S))$. Similarly if $S = \PSp_4(2^a)$ we take $s \in G^*$ with eigenvalues $(\omega,\omega,\omega^{-1},\omega^{-1},1)$ and the seventh paragraph of the proof of \ref{theo:exceptions} gives $\chi_s \in \Irr_{0,\sigma}(B_0(S))$.  By Lemma \ref{tilde} we have that in both cases $\chi_s$ extends to some $\hat\chi_s \in  \Irr(B_0(\tilde S))$. Now let $X$ be the preimage in $\tilde{S}A$ of a Sylow $3$-subgroup of $\tilde{S}A/\tilde{S}$. Then by \cite[Proposition 7.2]{Tay18} we see that $\hat{\chi}_s$ is $X$-invariant. In both cases we have $\tilde{S}/\tilde S'$ has $2$-power order. Thus, we have that the determinental order $o(\hat\chi_s)$ must be prime to $3$. Thus, \cite[Corollary 6.4]{Nav18} gives an extension of $\hat\chi_s$  in $X$ which must be $\sigma$-fixed. Then applying Lemma \ref{3prime} gives a character above $\chi_s$ in $\Irr_{0,\sigma}(B_0(A)) $, which must be distinct from those which lie above the unipotent characters because non-unipotent characters cannot be conjugate to unipotent characters. 
\end{proof}

Note that Theorem A follows from  Theorem \ref{allbutpsl3}, Proposition \ref{theo:psl2},
Theorem \ref{theo:exceptions},   Proposition \ref{psl3}, Proposition \ref{defchar},  Proposition \ref{gapstuff}, and Proposition \ref{An}.

\textbf{Acknowledgments}:
The work was completed as part of my Ph.D. study at the University of Denver. I would like to thank and acknowledge my advisor A. A. Schaeffer Fry for her support and guidance. I would also like to  acknowledge  J. Miquel Mart\'inez and Noelia Rizo for the helpful conversations we've had on this topic. I would also like that thank John McHugh and the anonymous referees for helpful comments, which improved the readability of the article.




\begin{thebibliography}{ABCDEF}


\bibitem[Alp76]{Alperin76}
{\sc J.\,L. Alperin,}
\newblock Isomorphic blocks,
\newblock {\em J. Algebra} \textbf{43} (1976), 694--698.

\bibitem[BCP97]{BCP97} W. Bosma, J.~J. Cannon and C. Playoust, The Magma algebra system. I. The user language, J. Symbolic Comput. {\bf 24} (1997), no.~3-4, 235--265.

\bibitem[Coo81]{Coo81} B.~N. Cooperstein, Maximal subgroups of $G\sb{2}(2\sp{n})$, J. Algebra {\bf 70} (1981), no.~1, 23--36. 

\bibitem[CE04]{CE04}
{\sc M. Cabanes and M. Enguehard}, \emph{Representation theory of
finite reductive groups}, New Mathematical Monographs \textbf{1},
Cambridge University Press, Cambridge, 2004.

\bibitem[Dad77]{Dad77}
{\sc E.\,C. Dade},
\newblock Remarks on isomorphic blocks,
\newblock {\em J. Algebra} \textbf{45} (1977), 254--258.

\bibitem[DM87]{DM87}
{\sc D. I. Deriziotis and G. O. Michler}, Character table and blocks of finite simple triality groups
$\,^{3}\operatorname{D}_4(q)$, \emph{ Trans. Amer. Math. Soc. {\bf 303}} (1987), 39--70.


\bibitem[FS82]{FS82}
{\sc P.~Fong and B.~Srinivasan},
\newblock The blocks of finite general linear and unitary groups,
\newblock {\em Invent. Math.} \textbf{69} (1982), 109--153.


\bibitem[FS89]{FS89}
{\sc P.~Fong and B.~Srinivasan},
\newblock The blocks of finite classical groups, 
\newblock {\em J. reine angew. Math. {\bf 396}} (1989), 122--191.



\bibitem[GAP]{gap} The GAP Group, GAP – Groups, Algorithms, and Programming, Version 4.11.0, 2020. http://www.gap-system.org.

\bibitem[Gec03]{Gec03} M. Geck. Character values, Schur indices, and character sheaves. Represent. Theory 7 (2003), 19–55.

\bibitem[GLS98]{GLS}
{\sc D. Gorenstein, R. Lyons, and R. Solomon}, \emph{The
classification of the finite simple groups}, Number 3, Part I,
Chapter A, Almost simple K-groups, Mathematical Surveys and Monographs, 40.3, American Mathematical Society, Providence, RI, 1998.

\bibitem[GM20]{GM20}
{\sc M. Geck, and G. Malle}, \emph{The character theory of finite groups of Lie
  type: A guided tour}. Cambridge University Press, Cambridge, 2020.


\bibitem[GRSV24]{sf}
{\sc E. Giannelli, N. Rizo, A. A. Schaeffer Fry,
AND C. Vallejo}, Characters and Sylow 3-Subgroup Abelianization \emph{J. Algebra \bf{667}} (2025), 824–864.

\bibitem[GSV19]{GSV}
{\sc E. Giannelli,  A.\,A. Schaeffer Fry, and C. Vallejo},
Characters of $\pi'$-degree, \emph{ Proc. Amer. Math. Soc. {\bf 147}}, no. 11 (2019), 4697--4712. 

\bibitem[HMM22]{HMM22}{\sc H.~N. Nguyen, G. Malle and A. Mar\'oti}, On almost $p$-rational characters of $p'$-degree, Forum Math. {\bf 34} (2022), no.~6, 1475--1496.

\bibitem[HiSh90]{HiSh90}
{\sc G. Hiss and J. Shamash}, $3$-blocks and $3$-modular characters of $\operatorname{G}_2(q)$, \emph{J. Algebra} {\bf 131} (1990), 371--387.


\bibitem[Kle88]{Kle88}P.~B. Kleidman, The maximal subgroups of the Chevalley groups $G_2(q)$ with $q$ odd, the Ree groups $^2G_2(q)$, and their automorphism groups, J. Algebra {\bf 117} (1988), no.~1, 30--71. 

\bibitem[Lub]{lubeckwebsite}
{\sc F.~L{\"u}beck}, Character Degrees and their Multiplicities for some Groups of Lie Type of Rank $< 9$, {https://www.math.rwth-aachen.de/~Frank.Luebeck/chev/DegMult/index.html}


\bibitem[Mal90]{Mal90}
{\sc G. Malle}, Die unipotenten Charaktere von $\,^2\operatorname{F}_4(q^2)$,
\emph{Comm. Algebra {\bf 18}} (1990), 2361--2381.

\bibitem[Mal08]{malle08}
{\sc G. Malle}, Extensions of unipotent characters and the inductive
McKay condition, \emph{J. Algebra} \textbf{320} (2008), 2963--2980.


\bibitem[Mal20]{mal20}
{\sc G. Malle}, On the number of characters in blocks of quasi-simple groups, \emph{Algebr. Represent. Theory {\bf 23}} (2020), 513--539.

\bibitem[MS22]{MS22}
{\sc A. Moret\'o and B. Sambale}, Groups with 2-generated Sylow subgroups and their character tables, \emph{ Pacific J. Math}. {\bf 323} (2023), no.~2, 337--358.

\bibitem[MT11]{MT11}
{\sc G. Malle and D. Testerman}, \emph{Linear algebraic groups and finite groups of Lie type}. Cambridge
Studies in Advanced Mathematics, 133, Cambridge University Press, Cambridge, 2011.

\bibitem[MMSV24]{mmsv}{\sc A. Mar\'oti, J. M. Marti\'nez, A. A. Schaeffer Fry, and C. Vallejo} On almost $p$-rational characters in principal blocks, \emph{ Publ. Mat.} {\bf 70} (2026), no.~1, 55--78.

\bibitem[MO83]{MO83}
{\sc G.\,O. Michler and J.\,B. Olsson}, Character correspondences in
finite general linear, unitary and symmetric groups, \emph{Math. Z.}
\textbf{184}, 203--233.

\bibitem[Nav98]{Nav98}
{\sc G.~Navarro},
\newblock {\em Characters and blocks of finite groups}, {London Mathematical Society Lecture Note Series} \textbf{50},
\newblock Cambridge University Press, Cambridge, 1998.


\bibitem[Nav18]{Nav18} { \sc G. Navarro~Ortega} {\it Character theory and the McKay conjecture}, Cambridge Studies in Advanced Mathematics, 175, Cambridge Univ. Press, Cambridge 2018.



\bibitem[NRSV21]{NRSV21}
{\sc G. Navarro, N. Rizo, A.\,A. Schaeffer Fry, and C. Vallejo}  Characters and generation of Sylow 2-subgroups, Represent. Theory {\bf 25} (2021), 142--165.

\bibitem[NT13]{NT13}
{\sc G. Navarro and P.\,H. Tiep},
Characters of relative $p'$-degree over normal subgroups, \emph{Ann. Math.} (2)
 {\bf 178:3} (2013) 1135--1171.


\bibitem[Ro96]{RobinsonGroupTheory}
{\sc D.J.S. Robinson}, \emph{A course in the theory of groups}, Graduate Texts in Mathematics Vol. 80, Springer, 1996.

\bibitem[RSV20]{RSV20}
{\sc N. Rizo, A.\,A. Schaeffer Fry, and C. Vallejo}, Galois action on the principal block and cyclic Sylow subgroups, \emph{ Algebra  Number Theory {\bf{14}}}:7 (2020), 1953--1979.

\bibitem[SF73]{SF73} {\sc W.A. ~Simpson, J.S. ~Frame}, The character tables for $\operatorname{SL}(3, q)$, $\operatorname{SU}(3, q^2)$, $\PSL(3, q)$, $\operatorname{PSU}(3,q^2)$. {\em Canad. J. Math.} \textbf{25} (1973), 486–494.

\bibitem[SV19]{fieldofvalues}
{\sc A.A.Schaeffer Fry and C. Ryan Vinroot}, Fields of Character Values For Finite Special Unitary Groups, \emph{Pacific J. Math. 300 (2019), no. \bf{2}}, 473–489.

\bibitem[Tay18]{Tay18}
{\sc J. Taylor},  Action of automorphisms on irreducible characters of symplectic
groups, \emph{J. Algebra {\bf 505}} (2018), 211--246.

\bibitem[Val23]{Val23} {\sc C. Vallejo~Rodr\'iguez}, A lower bound on the number of generators of a defect group, Vietnam J. Math. {\bf 51} (2023), no.~3, 571--576.

\bibitem[War66]{War66} H.~N. Ward, On Ree's series of simple groups, Trans. Amer. Math. Soc. {\bf 121} (1966), 62--89.

\bibitem[We55]{weir}
{\sc A. J. Weir}, Sylow $p$-subgroups of the classical groups over finite fields with characteristic prime to $p$, \emph{Proc. Amer. Math. Soc. {\bf 6}}, No. 4 (1955), 529--533.


\end{thebibliography}
\end{document}